\documentclass[12pt]{amsart}
\usepackage{amssymb,amsmath,color}
\def\eps {\varepsilon}
\def\d {{\rm d}}
\def\e {{\rm e}}
\def\R {\mathbb{R}}

\def\H {{\mathcal H}}
\def\h {{\rm H}}
\def\cbt {{\sc cbt}}

\def\BB {{{\mathbb B}}}
\def\C {{\mathcal C}}
\def\D {{\mathfrak D}}
\def\Q {{\mathcal Q}}
\def\EE {{\mathcal E}}
\def\LL {{\mathcal L}}

\def\T {{\mathbb T}}
\def\L {{\mathcal L}}
\def\M {{\mathcal M}}
\def \l {\langle}
\def \r {\rangle}

\def\ddt{\frac{\d}{\d t}}

\def \and {{\qquad\text{and}\qquad}}
\def\pt{\partial_t}

\def\ps{\partial_s}

\def\ptt{\partial_{tt}}
\def \di{{\rm dist}}
\def\ds {\displaystyle}
\def\Mr {{\rm M}}
\def\Ks{{\textsf{K}}}
\def \att{{\mathfrak A}}

\newtheorem{proposition}{Proposition}[section]
\newtheorem{theorem}[proposition]{Theorem}
\newtheorem{corollary}[proposition]{Corollary}
\newtheorem{lemma}[proposition]{Lemma}
\theoremstyle{definition}
\newtheorem{definition}[proposition]{Definition}
\newtheorem{remark}[proposition]{Remark}
\newtheorem*{example}{Example}
\numberwithin{equation}{section}

\def \au {\rm}
\def \ti {\it}
\def \jou {\rm}
\def \bk {\it}
\def \no#1#2#3 {{\bf #1} (#3), #2.}
\def \eds#1#2#3 {#1, #2, #3.}

\title[Viscoelasticity with time-dependent memory kernels]
{Viscoelasticity with time-dependent memory kernels.
Part II: asymptotic behavior of solutions}

\author[M. Conti, V. Danese and V. Pata]
{Monica Conti, Valeria Danese and Vittorino Pata}

\address{Politecnico di Milano - Dipartimento di Matematica
\newline\indent
Via Bonardi 9, 20133 Milano, Italy}
\email{monica.conti@polimi.it {\rm (M.\ Conti)}}
\email{valeria.danese@polimi.it {\rm (V.\ Danese)}}
\email{vittorino.pata@polimi.it {\rm (V.\ Pata)}}

\subjclass[2000]{35B41, 45K05, 73E50, 74D99}

\keywords{Viscoelasticity, Kelvin-Voigt model, memory, time-dependent kernels, processes on time-dependent spaces,
time-dependent global attractors}

\begin{document}

\begin{abstract}
We continue the analysis
on the model equation arising in the theory of viscoelasticity
$$
\ptt u(t)-\big[1+k_t(0)\big]\Delta u(t) -\int_0^\infty k'_t(s)\Delta u(t-s)\d s + f(u(t)) = g
$$
in the presence of a (convex, nonnegative and summable) memory kernel $k_t(\cdot)$ explicitly depending on time.
Such a model is apt to
describe, for instance, the dynamics of aging viscoelastic materials.
The earlier paper~\cite{Timex} was concerned with the correct mathematical setting of the problem,
and provided a well-posedness result within the novel theory of dynamical systems
acting on time-dependent spaces, recently established by Di Plinio {\it et al.}\ \cite{oscillon}.
In this second work, we focus on the asymptotic properties of the solutions,
proving the existence and the regularity of the time-dependent global attractor for
the dynamical process generated
by the equation. In addition, when $k_t$ approaches a multiple $m\delta_0$ of the Dirac mass at zero
as $t\to\infty$,
we show that the asymptotic dynamics of our problem is close to the one of
its formal limit
$$\ptt u(t)-\Delta u(t) -m\Delta\pt u(t)+ f(u(t)) = g$$
describing viscoelastic solids of Kelvin-Voigt type.
\end{abstract}

\maketitle

\section{Introduction}

\noindent
Given a bounded domain $\Omega \subset \R^3$ with smooth boundary
$\partial\Omega$, take the Hilbert space
$\h=L^2(\Omega)$, and let
$$A=-\Delta \quad\text{with domain}\quad \D(A)=H^2(\Omega)\cap H_0^1(\Omega)$$
be the Laplace-Dirichlet operator on
$\h$. In the earlier work \cite{Timex},
for any initial time $\tau \in \R$, we considered the evolution problem
in the unknown variable
$u=u(t):[\tau,\infty)\to \h$
\begin{equation}
\label{AAA}
\ptt u(t)+A u(t) +\int_0^\infty \mu_t(s)A\eta^t(s)\d s + f(u(t)) = g,
\end{equation}
\begin{equation}
\label{ETA}
\eta^t(s) =
\begin{cases}
u(t) - u(t-s),                     &\,0<s \leq t- \tau,\\
\eta_{\tau}(s-t+\tau) + u(t) - u_\tau, &\,s > t -\tau,
\end{cases}
\end{equation}
\begin{equation}
\label{CCC}
\begin{cases}
u(\tau) = u_\tau,\\
\pt u(\tau) = v_\tau,\\
\eta^\tau(s) =\eta_\tau(s),\quad s>0,
\end{cases}
\end{equation}
where the initial values $u_\tau$, $v_\tau$ and $\eta_\tau=\eta_\tau(s)$ are assigned data.
Here, $f(u)$ is a nonlinear term,
$g$ a constant-in-time external force, while
the function
$\mu_t=\mu_t(s)$ of the variable $s>0$ is the so-called memory kernel,
which is allowed to exhibit an explicit dependence on time, and is supposed to be absolutely continuous
on $\R^+=(0,\infty)$,
nonincreasing and summable (hence nonnegative) for every fixed $t$.
As discussed in detail in \cite{Timex}, problem \eqref{AAA}-\eqref{CCC}
arises in the theory of uniaxial deformations in isothermal viscoelasticity (see e.g.\ \cite{CHR,FM,RHN}).
In this context,
the auxiliary variable\footnote{The idea of describing the past values (i.e.\ the history) of $u$ via the introduction of
an auxiliary variable goes back to the pioneering papers of C.M.\ Dafermos \cite{DAF,DAF2}.}
$\eta=\eta^t(s)$ contains all the information on the past history (i.e.\ for times $t<\tau$) of
the axial displacement field $u$. Indeed, assuming $u$ to be known for all past times,
and interpreting the initial value $\eta_\tau$ as
$$\eta_\tau(s)=u(\tau)-u(\tau-s),$$
then \eqref{AAA}-\eqref{ETA} take the more familiar form\footnote{Here and in what follows,
the {\it prime} denotes the derivative with respect to the internal variable $s$.}
\begin{equation}
\label{mff}
\ptt u(t)+\big[1+k_t(0)\big]A u(t) +\int_0^\infty k'_t(s)Au(t-s)\d s + f(u(t)) = g,
\end{equation}
where the (nonnegative) convex function
\begin{equation}
\label{kkk}
k_t(s)=\int_s^\infty \mu_t(y)\d y
\end{equation}
is the integrated memory kernel, supposed to be summable as well, which satisfies by construction the relation
$$k'_t(s)=-\mu_t(s).$$
Since a formal integration by parts yields
\begin{equation}
\label{FIBP}
\int_{0}^\infty k_t'(s)A u(t-s)\d s
=-k_t(0)A u(t) + \int_0^\infty k_t(s)A\pt u(t-s)\d s,
\end{equation}
equation \eqref{mff} can be equivalently written as
\begin{equation}
\label{mff2}
\ptt u(t)+A u(t) +\int_0^\infty k_t(s)A\pt u(t-s)\d s + f(u(t)) = g.
\end{equation}

The main novelty of the model lies in the fact that the memory kernel depends itself on time.
This feature allows to describe viscoelastic materials whose structural properties evolve over time.
For instance, materials that undergo an aging process, which can be reasonably
depicted as a loss of the elastic response.
Indeed, the most interesting situation is when in the limit $t\to\infty$
the viscous effects become instantaneous.
Describing a viscoelastic solid through a rheological model
as a Maxwell element (i.e.\ a Hookean spring and a Newtonian dashpot sequentially connected)
in parallel with a lone spring, this translates
into a progressive stiffening of the spring in the Maxwell element,
becoming eventually completely rigid (see the discussion in~\cite{Timex}, see also \cite{Droz}).
In mathematical terms, this means that
$$k_t\to m\delta_0,\quad m>0,$$
namely, the integrated kernel $k_t(s)$ converges in the distributional
sense to
(a multiple of) the Dirac mass $\delta_0$ at zero as $t\to\infty$. In such a case,
equation \eqref{mff2} formally collapses in the longtime into the Kelvin-Voigt model of viscoelasticity
\begin{equation}
\label{SDWE}
\ptt u(t)+A u(t) +m A\pt u(t)+ f(u(t)) = g,
\end{equation}
where the memory term disappears.
In the terminology of Dautray and Lions \cite{DOLI},
this represents the transition from a viscoelastic solid with ``long memory" to
a viscoelastic solid with ``short memory", otherwise called of rate-type (see~\cite{DL}).
Notably, this convergence (at least at a formal level) occurs {\it within}
the dynamics, and not just by letting some parameters go artificially to zero.

The theoretical challenge in \cite{Timex} was to figure out a correct mathematical setting of the problem,
in order to obtain an existence and uniqueness result.
Indeed, the presence of a time-dependent kernel introduces essential difficulties in the analysis.
As an example,
to make a comparison with the standard theory,
in the original work of Dafermos \cite{DAF} (as well as in a number of later papers) the auxiliary variable $\eta$
is ruled by the differential equation
$$\pt \eta^t(s)=-\partial_s\eta^t(s)+\pt u(t).$$
Contrary to the case of a memory kernel independent of time,
here the problem is that the natural space of $\eta^t$ depends itself on $t$, reason why it is not even clear how
to define the time-derivative of the variable. Nonetheless, as shown in \cite{Timex}, it is possible to give
a good definition of solution within the theory of processes on time-dependent spaces,
recently devised by Di Plinio {\it et al.}\ \cite{oscillon}, and further developed in \cite{calimero,katilina,CPT,oscillon2}.

\smallskip
The aim of the present work is to continue the analysis started in \cite{Timex},
studying the asymptotic properties of the solutions from a global-geometrical point of view.

\smallskip
\noindent
$\bullet$
First, we establish the dissipative character of
the dynamical process generated by \eqref{AAA}-\eqref{CCC}, acting on suitable time-dependent phase spaces $\H_t$.
By means of a recursive argument, we are able to find
a constant $r_0>0$, independent of $t$ and $\tau$,
such that the evolution at time $t$
of $\H_\tau$-bounded sets of initial data assigned at time $\tau<t$
has $\H_t$-norm less than $r_0$ as $t-\tau\to\infty$. In the terminology of \cite{CPT},
this entails the (uniform) time-dependent absorbing set.
Incidentally, the result applies (and is new) also for the classical case of a constant-in-time memory kernel,
where an explicit stabilization estimate for the energy was not available.
Indeed, e.g.\ in \cite{Visco,GMPZ},
the absorbing set was
eventually recovered as a byproduct of the global attractor, which in turn
was obtained relying on the existence
of a Lyapunov functional.
Instead, direct energy estimates have two advantages: from the one side, the actual
entering time of the solutions into the absorber can be calculated, from the other side
the known results on the viscoelastic model can be extended, allowing for instance
the presence of a time-dependent forcing term, that destroys the
gradient system structure of the equation.

\smallskip
\noindent
$\bullet$
The second step towards a more refined analysis is showing that the dynamical process
possesses the so-called time-dependent global attractor. Loosely speaking, this is the
smallest family of $t$-labeled sets $A_t$ able to attract
bounded sets of initial data in a pullback sense.
Besides, we will prove the optimal regularity of such a family.
The existence of the time-dependent global attractor, which for the model under consideration
turns out to be also invariant under the action of the process,
provides a complete characterization of the regime behavior of the solutions.

\smallskip
\noindent
$\bullet$
The final goal is to understand what happens in the limit $k_t\to m\delta_0$.
In that situation, the convergence of the original problem to \eqref{SDWE} turns out
to be not only formal.
Indeed, we will show that the $t$-sections $A_t$ of the time-dependent global attractor
satisfy in a suitable sense the relation
$$A_t\to \hat A,\quad\text{as }\,t\to\infty,$$
where $\hat A$ is the global attractor of the semigroup
generated by~\eqref{SDWE}, as defined in the classical books \cite{BV,HAL,HAR,TEM}.
This provides a rigorous proof of the closeness of
\eqref{AAA}-\eqref{CCC} to the Kelvin-Voigt model when the integrated kernel $k_t$ approaches
the Dirac mass at zero as $t\to\infty$.

\subsection*{Outline of the paper}
Under the general assumptions of the next Section~\ref{SGA},
we first recall the main result of \cite{Timex} on the generation of the time-dependent
process of solutions. This is done in
Section~\ref{SPTDS}. Our results on the dissipative character of the process
are stated in Section~\ref{SSR}. The subsequent Sections~\ref{SSAF}-\ref{SKV}
are devoted to their proofs. Namely, in Section~\ref{SSAF} we introduce some auxiliary
energy functionals, establishing suitable integral inequalities (the proof of one of which being
quite long and technical, and therefore postponed in the Appendix A).
In Section~\ref{SDIS} we prove the existence of the time-dependent
absorbing set, while the proofs concerning the time-dependent global attractor
are given in Section~\ref{SATT} (existence), and Section~\ref{SREG} (regularity).
Here a key ingredient is a Gronwall-type lemma in integral form, discussed in the Appendix B.
In Section~\ref{SKV} we show the asymptotic closeness to the Kelvin-Voigt model when the kernel
approaches the Dirac mass.
In the final Section~\ref{STMK} we dwell on two particular memory kernels of physical interest
complying with our general assumptions.

\subsection*{Notation}
For $\sigma \in \R$, we define the compactly nested Hilbert spaces
$$\h^\sigma = \D(A^{\sigma/2}),\qquad
\l u,v\r_{\sigma}
=\l A^{\sigma/2}u,A^{\sigma/2}v\r_{\h}.$$
Throughout the paper, the index $\sigma$ will be always omitted whenever zero.
The symbol $\l\cdot,\cdot\r$ will also be used to denote the duality pairing between
$\h^\sigma$ and its dual space $\h^{-\sigma}$.
Then, for every fixed time $t$,
we introduce the weighted $L^2$-spaces, hereafter called memory spaces,
$$\M_t^{\sigma} = L^2_{\mu_t}(\R^+; \h^{\sigma+1}),\qquad
\l\eta,\xi\r_{\M_t^{\sigma}}
=\int_0^\infty \mu_t(s)\l\eta(s),\xi(s)\r_{\sigma+1}\d s.$$
We will also consider the linear operator acting on $\M_t^\sigma$
$$\T_t\eta=-\eta'
\quad\text{with domain}\quad
\D(\T_{t})=\big\{\eta\in{\M^\sigma_t}:\,
\eta'\in\M^\sigma_t,\,\,
\eta(0)=0\big\}.$$
Finally, we define the extended memory spaces
$$\H^\sigma_t = \h^{\sigma+1} \times \h^\sigma \times \M_t^{\sigma},$$
endowed with the usual product norm.
For any $r\geq 0$, we will denote by
$$\BB_t^\sigma(r)=\big\{z\in \H_t^\sigma:\,\|z\|_{\H_t^\sigma}\leq r\,\big\}$$
the closed $r$-ball about zero of $\H_t^\sigma$.

\section{General Assumptions}
\label{SGA}

\noindent
We begin to stipulate our assumptions on the external force $g$, the nonlinear term $f(u)$
and the memory kernel $\mu_t$.

\subsection{Assumptions on $\boldsymbol{g}$ and $\boldsymbol{f}$}
Let $g\in \h$ be independent of time, and  let $f\in \mathcal C^2(\R)$, with
$f(0)=0$, satisfy for some $c\geq 0$
\begin{equation}
\label{hp1-f}
|f''(u)| \leq c(1+|u|),
\end{equation}
along with the dissipation condition
\begin{equation}
\label{hp2-f}
\liminf_{|u|\to\infty}f'(u)>-\lambda_1,
\end{equation}
$\lambda_1>0$ being the first eigenvalue of $A$.
In particular, \eqref{hp2-f} easily implies the relations
\begin{align}
\label{phi-funz1}
&2\l F(u),1\r\geq -(1-\theta)\|u\|^2_1-c_f,\\
\label{phi-funz2}
&2\l f(u),u\r\geq 2\l F(u),1\r -(1-\theta)\|u\|^2_1-c_f,
\end{align}
for some $0<\theta\leq 1$ and $c_f \geq 0$,
where
$$F(u)=\int_0^u f(s)\d s.$$

\begin{remark}
\label{remzeroo}
If $f$ is essentially monotone, i.e.\
\begin{equation}
\label{hp2bis-f}
\inf_{u\in\R}f'(u)>-\lambda_1,
\end{equation}
then $c_f=0$ in formulae \eqref{phi-funz1}-\eqref{phi-funz2}.
\end{remark}

\subsection{Assumptions on the memory kernel}
The map
$$(t,s)\mapsto\mu_t(s):\R\times\R^+\to\R^+$$
satisfies the following axioms.

\begin{itemize}
\item[{\bf (M1)}]
For every fixed $t\in\R$, the map
$s\mapsto\mu_t(s)$ is nonincreasing,
absolutely continuous and summable.
We denote the total mass of $\mu_t$ by
$$
\kappa(t) = \int_0^\infty \mu_t(s)\d s.
$$

\smallskip
\item[{\bf (M2)}]
For every $\tau\in\R$, there exists a function $K_\tau:[\tau,\infty)\to\R^+$,
summable on any interval $[\tau,T]$,
such that
$$
\mu_t(s) \leq K_\tau(t)\mu_\tau(s)
$$
for every $t\geq\tau$ and every $s>0$.

\smallskip
\item[{\bf (M3)}]
For almost every fixed $s>0$,
the map $t\mapsto\mu_t(s)$ is differentiable for all $t\in\R$, and\footnote{Here and in what follows,
the {\it dot} denotes the derivative with respect to time.}
$$(t,s)\mapsto \mu_t(s)\in L^\infty({\mathcal K})\and (t,s)\mapsto \dot\mu_t(s)\in L^\infty({\mathcal K})$$
for every compact set ${\mathcal K}\subset \R\times\R^+$.

\smallskip
\item[{\bf (M4)}]
There exists $\delta>0$ such that
$$\dot\mu_t(s)+\mu_t'(s)+\delta\kappa(t)\mu_t(s)\leq 0$$
for every $t\in\R$ and almost every $s>0$.

\smallskip
\item[{\bf (M5)}]
The function $t\mapsto\kappa(t)$ fulfills
$$\inf_{t\in \R} \kappa(t) > 0.$$

\smallskip
\item[{\bf (M6)}]
The function $t\mapsto\dot\mu_t(s)$ satisfies the uniform integral estimate
$$\sup_{t\in \R}\frac1{[\kappa(t)]^2}\int_0^\infty
|\dot\mu_t(s)|\d s <\infty.$$

\item[{\bf (M7)}]
For every $t\in\R$, the function $s\mapsto\mu_t(s)$ is bounded about zero, with
$$\sup_{t\in \R} \frac{\mu_t(0)}{[\kappa(t)]^2} <\infty.$$

\smallskip
\item[{\bf (M8)}]
For every $a<b\in\R$, there exists $\nu>0$ such that
$$\int_\nu^{1/\nu}\mu_t(s)\d s\geq \frac{\kappa(t)}{2}$$
for every $t\in[a,b]$.
\end{itemize}

\subsection{About axioms (M1)-(M8)}
Axioms {\bf (M1)}-{\bf (M4)} have been introduced in \cite{Timex}, and turn
out to be sufficient in order to obtain a well-posedness result (see the next Theorem~\ref{thm-ex-un}). Actually, such
a result in \cite{Timex} holds by replacing {\bf (M4)} with the weaker
\begin{itemize}
\smallskip
\item[{\bf (M4)$'$}]
There exists a function $M:\R\to\R^+$,
bounded on bounded intervals, such that
$$
\dot\mu_t(s)+\mu_t'(s)\leq M(t)\mu_t(s)
$$
for every $t\in\R$ and almost every $s>0$.
\end{itemize}

\noindent
Note that
in the classical case where the kernel does not depend on $t$ (i.e.\ $\mu_t= \mu$ for all $t$), axiom
{\bf (M4)} boils down to the well-known assumption
$$\mu'(s)+\delta\mu(s)\leq 0,$$
devised in~\cite{DAF} and commonly adopted in the literature thereafter.
Instead, the role of {\bf (M5)}-{\bf (M8)} is more technical, and it will become clear in the Appendix B.
Actually, it is also possible to treat the case of kernels which are unbounded about
zero, by weakening axiom {\bf (M7)}. We prefer to avoid such a choice here, which would introduce annoying
(and unnecessary) complications, see e.g.\ \cite{PAT}. We also point out that we are considering
kernels that do not vanish on $\R^+$, modeling
the so-called infinite delay case. However, without changes in the proofs, our analysis
apply as well to the finite delay case, namely, when $\mu_t(s)$ becomes identically zero for $s$ large enough.

\smallskip
\noindent
We conclude with
some immediate consequences of the axioms that will be useful in the course of the investigation.

\smallskip
\noindent
$\bullet$
In the light of {\bf (M1)}, for every fixed $t$ the function $s\mapsto\mu_t(s)$ is differentiable almost
everywhere with
$\mu_t'\leq 0$. In particular (see e.g.\ \cite{Terreni}),
$$
\l \T_t\eta, \eta\r_{\M_t^{\sigma}}
= \frac12\int_0^\infty \mu'_t(s) \|\eta(s)\|^2_{\sigma+1}\d s\leq 0 ,\quad\forall \eta\in\D(\T_t).
$$
This tells that $\T_t$ is a dissipative operator. Indeed, $\T_t$ turns out to be the
infinitesimal generator of the right-translation semigroup
on $\M_t^{\sigma}$.

\smallskip
\noindent
$\bullet$
For every $\sigma\in\R$ and every $t>\tau$, axiom {\bf (M2)} entails
$$
\|\eta\|^2_{\M_t^{\sigma}}
\leq K_\tau(t)\|\eta\|^2_{\M_\tau^{\sigma}},
\quad\forall \eta \in \M_\tau^{\sigma},
$$
providing the continuous embedding
$\M_\tau^{\sigma} \subset \M_t^{\sigma}$, hence $\H_\tau^{\sigma} \subset \H_t^{\sigma}$.
In particular, $\T_t \supset \T_\tau$, i.e.\
the operators $\{\T_t\}_{t\geq\tau}$ are
increasingly nested extensions of each other.

\section{The Process on Time-Dependent Spaces}
\label{SPTDS}

\noindent
Let us begin with the definition of weak solution from \cite{Timex}.

\begin{definition}
\label{def-sol}
Let $T>\tau \in \R$, and let $z_\tau=(u_\tau,v_\tau,\eta_\tau)\in\H_\tau$ be a fixed vector.
A function
$$z(t) = (u(t), \pt u(t), \eta^t)\in\H_t \quad\text{for a.e. } t\in [\tau,T]$$
is a solution
to problem \eqref{AAA}-\eqref{CCC}
on the time-interval $[\tau, T]$ with initial datum  $z_\tau$ if:
\smallskip
\begin{itemize}
\item[{\rm (i)}] $u \in L^\infty(\tau, T; \h^1)$, $\,\pt u\in  L^\infty(\tau, T; \h)$, $\,\ptt u\in  L^1(\tau, T; \h^{-1})$.
\smallskip
\item[{\rm (ii)}] $u(\tau)=u_\tau$, $\,\pt u(\tau)=v_\tau$.
\smallskip
\item[{\rm (iii)}] The function $\eta^t$ fulfills the representation formula~\eqref{ETA}.
\smallskip
\item[{\rm (iv)}] The function $u(t)$ fulfills \eqref{AAA} in the weak sense,
i.e.\
$$\l\ptt u(t),\phi\r +\l u(t),\phi\r_1 +\int_0^\infty \mu_t(s)\l\eta^t(s),\phi\r_1\d s
+ \l f(u(t)),\phi\r = \l g,\phi\r$$
for almost every $t \in [\tau,T]$ and every test $\phi\in\h^1$.
\end{itemize}
\end{definition}

The main result of \cite{Timex} is the generation of a process of solutions
for problem~\eqref{AAA}-\eqref{CCC}.
Recall that a two-parameter family of operators
$$U(t,\tau):\H_\tau\to\H_t,\quad t\geq\tau,$$
is called a {\it processes on time-dependent spaces}
(see \cite{calimero,katilina,CPT,oscillon,oscillon2}) if
\begin{itemize}
\smallskip
\item[$\diamond$] $U(\tau,\tau)$ is the identity map on ${\H_\tau}$ for every $\tau$;
\smallskip
\item[$\diamond$] $U(t,\tau)U(\tau,s)=U(t,s)$ for every $t\geq\tau\geq s$.
\smallskip
\end{itemize}
This follows from (see \cite{Timex})

\begin{theorem}
\label{thm-ex-un}
For every $T>\tau \in \R$ and every initial datum
$z_\tau = (u_\tau, v_\tau, \eta_\tau) \in \H_\tau$,
problem \eqref{AAA}-\eqref{CCC} admits a unique solution
$$z(t)=(u(t), \pt u(t), \eta^t)=U(t,\tau)z_\tau$$
on the interval $[\tau, T]$.
Besides,
$u  \in \C([\tau, T], \h^1) \cap \C^1([\tau, T], \h)$,
$\eta^t \in\M_t$ for every $t$, and
$$\sup_{t\in[\tau,T]}\|U(t,\tau)z_\tau\|_{\H_t}<C,$$
for some $C>0$ depending only on $T,\tau$ and the $\H_\tau$-norm of $z_\tau$.
Moreover, the map
$$z_\tau\mapsto U(t,\tau)z_\tau$$
is (locally) Lipschitz from $\H_\tau$ to $\H_t$, uniformly with respect to $t\in[\tau,T]$.
\end{theorem}

For a given $z_\tau\in\H_\tau$, we agree to define the energy
of the solution $U(t,\tau)z_\tau$ as
$$
\EE(t,\tau) =\frac12 \|U(t,\tau)z_\tau\|^2_{\H_t}.
$$

\section{Statements of the Results}
\label{SSR}

\subsection{Dissipativity}
First, we discuss the dissipative character of the process. Mathematically speaking, this
means that $U(t,\tau)$
possesses a time-dependent absorbing set, as defined in~\cite{CPT}
(see also \cite{oscillon,oscillon2}).

\begin{definition}
\label{d-absorbing}
A family $\mathfrak B=\{B_t\}_{t\in\R}$ is called a \emph{(uniform) time-dependent absorbing set}
if it is \emph{uniformly bounded}, i.e.
$$\sup_{t\in\R}\|B_t\|_{\H_t}<\infty,$$
and, for every $R>0$, there exists
an elapsed entering time $\tau_\e=\tau_\e(R)\geq0$ such that
$$
t-\tau\geq \tau_\e
\quad\Rightarrow\quad
U(t,\tau)\BB_\tau(R)\subset B_t.
$$
\end{definition}

The existence of the time-dependent absorbing set is an immediate consequence
of the following result.

\begin{theorem}
\label{MAIN-ABS}
There exist constants $\omega>0$, $R_0\geq 0$ and an increasing positive function $\Q$,
all independent of $t\geq \tau$, such that
$$\EE(t,\tau) \leq \Q(R)\e^{-\omega(t-\tau)} + R_0,
$$
whenever $\EE(\tau,\tau)\leq R$.
\end{theorem}

Indeed, after Theorem \ref{MAIN-ABS}, Definition \ref{d-absorbing} applies
by merely taking
$$B_t=\BB_t(r_0)
\quad\text{with}\quad r_0>\sqrt{2R_0}\,.$$
In absence of a forcing term, and for $f$ essentially monotone,
the theorem holds with $R_0=0$,
yielding the exponential decay of the energy.

\begin{corollary}
\label{corMAIN-ABS}
If in addition $g=0$ and $f$ satisfies \eqref{hp2bis-f}, then
$$\EE(t,\tau) \leq \Q(R)\e^{-\omega(t-\tau)},
$$
whenever $\EE(\tau,\tau)\leq R$.
\end{corollary}

\subsection{The global attractor}
We then deepen the longterm analysis of the system,
looking for the time-dependent attractor. This is the object characterizing the regime behavior
of a process defined on a time-dependent family of spaces~\cite{CPT,calimero,oscillon, oscillon2}.

\begin{definition}
\label{def-pulla}
The {\it time-dependent global attractor} for $U(t,\tau)$ is the smallest
family $\att=\{A_t\}_{t\in\R}$
with the following properties:
\begin{itemize}
\item[{\rm (i)}] Each section $A_t$ is compact in $\H_t$.
\item[{\rm (ii)}] $\att$ is \emph{pullback attracting}, namely, it is uniformly bounded and
the limit\footnote{We denote the Hausdorff semidistance
of two (nonempty) sets $B,C\subset \H_t$ by
$$\di_{\H_t}(B,C)=\sup_{x\in B}\,\inf_{y\in C}\,\|x-y\|_{\H_t}.$$}
$$\lim_{\tau\to-\infty}\Big[\di_{\H_t}\big(U(t,\tau) C_\tau,A_t\big)\Big]=0$$
holds for every uniformly bounded family $\mathfrak C=\{C_t\}_{t\in\R}$ and every $t\in\R$.
\end{itemize}
\end{definition}

The existence of the (invariant) time-dependent global attractor for our problem
reads as follows.

\begin{theorem}
\label{MAIN-ATTRACTOR}
The process $U(t,\tau):\H_\tau\to\H_t$ possesses
the time-dependent global attractor $\att=\{A_t\}_{t\in\R}$. Besides, the attractor is invariant,
i.e.\
$$U(t,\tau)A_\tau=A_t, \quad \forall t\geq \tau.$$
\end{theorem}

According to \cite[Theorem 3.2]{calimero},
the invariant time-dependent global attractor is characterized as
the set of all {\it complete bounded trajectories}
({\cbt}) of the process, that is,
$$A_t=\big\{z(t): \,\text{$z$ {\cbt} of $U(t,\tau)$}\big\},$$
where a {\cbt} of $U(t,\tau)$ is a map
$$t\mapsto z(t)=(u(t),\pt u(t),\eta^t)\in \H_t$$
satisfying
$$
\sup_{t\in\R}\|z(t)\|_{\H_t}<\infty \and z(t)=U(t,\tau)z(\tau)\,\,\,\, \forall t\geq \tau\in\R.
$$

\begin{corollary}
\label{Corsquaz}
Given any {\cbt} $(u,\pt u,\eta)$, the equality
\begin{equation}
\label{squaz}
\eta^t(s)=u(t-s)-u(t)
\end{equation}
holds for every $t\in\R$ and every $s>0$.
In particular, \eqref{AAA}-\eqref{ETA} take the form~\eqref{mff}.
\end{corollary}

Indeed, being $(u,\pt u,\eta)$ a {\cbt}, for any fixed $t>\tau$
we readily see from~\eqref{ETA} that such an equality is true
for all $0<s\leq t-\tau$,
and letting $\tau\to-\infty$ the claim follows.

\subsection{Regularity}
The next theorem is concerned with the regularity of the attractor.

\begin{theorem}
\label{regMAIN-ATTRACTOR}
The sections of the time-dependent global attractor $\att=\{A_t\}_{t\in\R}$
(belong to and) are
uniformly bounded in $\H^1_t$, namely,
$$\sup_{t\in\R}\|A_t\|_{\H_t^1}<\infty.$$
\end{theorem}

\begin{remark}
\label{NEWEQ}
For any {\cbt} $(u,\pt u,\eta)$,
equality \eqref{squaz} together with the regularity of Theorem~\ref{regMAIN-ATTRACTOR}
ensure that
$$\eta^t\in\D(\T_t),\quad\forall t\in\R.
$$
Besides,
the formal integration by parts~\eqref{FIBP}, and so
the passage from~\eqref{mff} to~\eqref{mff2}, becomes rigorous.
Thus, the function $u(t)$ satisfies~\eqref{mff2}
in the weak sense for almost every $t\in\R$.
\end{remark}

\subsection{Recovering Kelvin-Voigt}
We finally discuss the case when
$k_t\to m\delta_0$ for some $m>0$, that is,\footnote{Condition \eqref{deltaconv}
is the same as saying that the measure $k_t(s)\d s$ on $[0,\infty)$ converges weakly to the measure
$m\delta_0$.}
\begin{equation}
\label{deltaconv}
\lim_{t\to\infty}\int_\eps^\infty k_t(s)\d s=
\begin{cases}
m &\text{if }\eps=0,\\
0 &\text{if }\eps>0.
\end{cases}
\end{equation}
Accordingly, in the longtime our problem collapses into
the Kelvin-Voigt model of viscoelastic solids~\eqref{SDWE}.

It is well-known that equation~\eqref{SDWE}, often referred to as
{\it strongly damped wave equation}, generates a $\C_0$-semigroup
of solutions
$$S(t):\h^1\times \h\to \h^1\times \h,$$
possessing the global attractor $\hat A$ in the classical sense.
Besides, $\hat A$ is a bounded subset of $\h^2\times \h^1$, and
coincides with the sections at (any) time $t_0\in\R$ of the set of all {\cbt} of $S(t)$
(see, e.g.\ \cite{carv-chol,EdenKala,ghimarz,massat,pasqu,PZsdwe,webb}). Namely, for any fixed $t_0\in\R$,
$$\hat A=\big\{\hat z(t_0):\,\text{$\hat z$ {\cbt} of $S(t)$}\big\}.$$
Recall that
a {\cbt} of the semigroup $S(t)$ is a map  (see \cite{HAR})
$$t\mapsto \hat z(t)=(\hat u(t),\pt\hat u(t))\in \h^1\times \h$$
satisfying
$$
\sup_{t\in\R}\|\hat z(t)\|_{\h^1\times \h}<\infty \and \hat z(t+\tau)=S(t)\hat z(\tau)\,\,\,\,
\forall t\geq 0,\,\forall\tau\in\R.
$$

Our last theorem establishes the closeness of the longterm dynamics of \eqref{AAA}-\eqref{CCC}
to the one of the ``limit problem" \eqref{SDWE} when $k_t\to m\delta_0$.

\begin{theorem}
\label{limitece}
Let \eqref{deltaconv} hold. Then,
for any sequence $(u_n,\partial_t u_n, \eta_n)$ of {\cbt} of $U(t,\tau)$ and any $t_n\to\infty$,
there exists a {\cbt} $(\hat u,\partial_t \hat u)$ of $S(t)$ such that
the convergence
\begin{equation}
\label{conv-cc}
\sup_{t\in [-T,T]}\,\Big[\|u_n(t+t_n)-\hat u(t)\|_{\h^1}+\|\partial_t u_n(t+t_n)-\partial_t \hat u(t)\|_{\h}\Big]\to 0
\end{equation}
holds up to a subsequence as $n\to\infty$ for every $T>0$.
\end{theorem}

Defining the canonical projection from $\H_t$ onto $\h^1\times \h$ by
${\mathbb P}_t (u,v,\eta)=(u,v)$, the theorem above produces an immediate corollary.

\begin{corollary}
If \eqref{deltaconv} holds, then we have the convergence
$$\lim_{t\to\infty}\Big[\di_{\h^2\times \h^1}\big({\mathbb P}_t A_t,\hat A\big)\Big]=0.$$
\end{corollary}

Indeed, \eqref{conv-cc} says in particular that for every $t_n\to\infty$
the convergence
$$\di_{\h^2\times \h^1}\big({\mathbb P}_{t_n} A_{t_n},\hat A\big)\to 0$$
holds (up to a subsequence) as $n\to\infty$. This is clearly enough to draw the desired conclusion.

\medskip
The proofs of the results stated above will be carried out in the next Sections \ref{SSAF}-\ref{SKV}.

\subsection*{A word of warning}
In the forthcoming proofs, we will denote by $C$ and $\Q$ a {\it generic} positive constant
and a {\it generic} increasing positive function, respectively, both independent of $t\geq \tau$.
We will use several times (possibly without explicit mention)
the Young, H\"older and
Poincar\'e inequalities,
as well as the standard Sobolev embeddings, such as  $\h^1\subset L^6(\Omega)$.
Besides, we will perform several energy estimates, which are rigorously justified within
the Galerkin approximation scheme detailed in \cite{Timex}.

\section{Some Auxiliary Functionals}
\label{SSAF}

\noindent
As we said in the Introduction, one of the main technical difficulties inherent in
the formulation of our problem is that we do not (cannot) have a differential
equation ruling the evolution of the additional variable $\eta^t$. As a direct consequence,
we are unable to draw directly differential inequalities,
essential to produce any kind of energy estimates. The strategy to overcome this obstacle
is to produce suitable integral inequalities, that would hold {\it if we had}
suitable differential estimates (which we don't).
To this end,
let $\tau\in\R$ be fixed and let $(p_\tau,q_\tau,\psi_\tau)\in \H_\tau$ be a sufficiently
regular initial datum.
We consider for $t>\tau$ the equation
\begin{equation}
\label{sis-p-psi}
\ptt p(t) + Ap(t) + \ds{\int_0^\infty \mu_t(s) A\psi^t(s)\d s}+ \gamma(t) = 0.
\end{equation}
Here, $\gamma$ is a certain forcing term (possibly depending on $p$), while
\begin{equation}
\label{rep-p-psi}
\psi^t(s) =
\begin{cases}
p(t) - p(t-s),                     &s \leq t- \tau,\\
\psi_{\tau}(s-t+\tau) + p(t) - p_\tau, &s > t -\tau.
\end{cases}
\end{equation}
The equation is supplemented with the initial conditions
\begin{equation}
\label{in-cond-p}
\begin{cases}
p(\tau)=p_\tau,\\
\pt p(\tau)=q_\tau,\\
\psi^\tau=\psi_\tau.
\end{cases}
\end{equation}
Assuming that \eqref{sis-p-psi}-\eqref{in-cond-p} admits a sufficiently regular global solution
$$(p(t),\pt p(t),\psi^t)\in \H_t$$
on $[\tau,\infty)$, we establish some crucial integral inequalities
involving the triplet $(p,\pt p,\psi)$. The first one, implied by
\eqref{rep-p-psi}-\eqref{in-cond-p} only, comes from \cite[Section 5]{Timex}.

\begin{lemma}
\label{lemma-eta-norm}
For all $\sigma\in [0,1]$ and every $b>a\geq\tau$,
we have
$$
\|\psi^b\|^2_{\M_b^{\sigma}}
-\int_a^b\!\!\int_0^\infty\big[\dot\mu_t(s)+\mu'_t(s)\big]\|\psi^t(s)\|^2_{\sigma+1}\d s\,\d t
\leq \|\psi^a\|^2_{\M_a^{\sigma}}
+2\int_a^b\l\pt p(t),\psi^t\r_{\M_t^{\sigma}}\d t.
$$
\end{lemma}

\begin{remark}
\label{int-eta}
On account of {\bf (M4)}, we deduce
in particular the integrability on $[a,b]$ of the map
$$t\mapsto \kappa(t)\|\psi^t\|^2_{\M_t^{\sigma}}.$$
\end{remark}

In the next two lemmas we state two further integral inequalities, in terms of
the auxiliary functionals
\begin{align}
\label{PHI}
\Phi(t) &= 2\l p(t),\pt p(t) \r, \\
\label{PSI}
\Psi(t) &= -\frac{2}{\kappa(t)}
\int_0^\infty\mu_t(s)\l\psi^t(s),\pt p(t)\r\d s,
\end{align}
which are easily seen to satisfy the estimate
\begin{equation}
\label{contr-Phi-Psi}
|\Phi(t)| + |\Psi(t)|
\leq C\big[\|p(t)\|^2_1 + \|\pt p(t)\|^2 + \|\psi^t\|^2_{\M_t}\big].
\end{equation}
Indeed, the less obvious control of $|\Psi|$ follows from the H\"older-type inequality
(which will actually occur several times in the forthcoming calculations)
$$\int_0^\infty \mu_t(s)\|\psi^t(s)\|_1\d s
\leq \sqrt{\kappa(t)}\|\psi^t\|_{\M_t},
$$
together with an application of {\bf (M5)}.

\begin{lemma}
\label{lemma-Phi-inc}
For every $b>a\geq\tau$ and every $\varpi \in (0,1]$, the functional $\Phi$ satisfies
\begin{align*}
\Phi(b) +
(2-\varpi)\int_a^b\|p(t) \|^2_1\d t
&\leq \Phi(a)+2\int_a^b\| \pt p(t) \|^2\d t\\
&\quad + \frac1\varpi\int_a^b\kappa(t)\| \psi^t \|^2_{\M_t}\d t
- 2\int_a^b\l\gamma(t),p(t)\r\d t.
\end{align*}
\end{lemma}

\begin{lemma}
\label{lemma-Psi}
For every $b>a\geq\tau$ and every $\varpi \in (0,1]$,
the functional $\Psi$ satisfies
$$\begin{aligned}
\Psi(b) + \int_a^b\|\pt p(t)\|^2\d t
& \leq \Psi(a) -M\int_a^b
\int_0^\infty[\dot\mu_t(s)+\mu_t'(s)]
\|\psi^t(s)\|^2_1\d s\,\d t\\
&\quad + \varpi\int_a^b\| p(t) \|^2_1\d t
+\frac{C}{\varpi}\int_a^b\kappa(t)\|\psi^t \|^2_{\M_t}\d t\\
&\quad+\int_a^b\frac{2}{\kappa(t)}
\int_0^\infty\mu_t(s)\l\psi^t(s),\gamma(t)\r\d s\,\d t.
\end{aligned}$$
The positive constants $M$ and $C$ depend only on the structural assumptions on the memory kernel.
\end{lemma}

The proof of Lemma \ref{lemma-Psi} is very technical and requires several approximation steps; it is therefore postponed in the
Appendix B in full detail.
On the contrary, the rather standard proof of the first lemma is reported here below for the reader's convenience.

\begin{proof}[Proof of Lemma \ref{lemma-Phi-inc}]
Multiplying equation \eqref{sis-p-psi} by $2p$ in
$\h$ we draw
$$
\ddt\Phi(t) - 2\|\pt p(t)\|^2
+ 2\|p(t)\|^2_1 + 2\l\gamma(t),p(t)\r
= - 2\int_0^\infty\mu_t(s)\l\psi^t(s),p(t)\r_1\d s.
$$
Estimating the right-hand side by
$$
- 2\int_0^\infty \mu_t(s)\l \psi(s), p (t)\r_1 \d s
\leq \varpi\|p(t)\|_1^2
+ \frac1\varpi\kappa(t)\|\psi^t\|^2_{\M_t},
$$
and integrating on $[a,b]$ the claim follows.
\end{proof}

Finally, we tailor the inequalities of Lemmas \ref{lemma-Phi-inc} and \ref{lemma-Psi}
to the special case where
$$
\gamma = f(p) - g,
$$
with $f$ and $g$ as in our problem.

\begin{lemma}
\label{lemma-Phi-Psi}
For the particular $\gamma$ above, the previous inequalities enhance to
\begin{align}
\label{p-mult-bis}
&\Phi(b) +
\bigg(1 + \frac\theta2\bigg)\int_a^b\|p(t) \|^2_1\d t
+ 2\int_a^b\l F(p(t)),1\r\d t - 2\int_a^b\l g,p(t)\r\d t \\
\notag
&\quad \leq \Phi(a)+2\int_a^b\| \pt p(t) \|^2\d t
+ \frac{2}{\theta}\int_a^b\kappa(t)\| \psi^t \|^2_{\M_t}\d t  + c_f(b-a),
\end{align}
with $\theta$ and $c_f$ given by \eqref{phi-funz2}, and
\begin{align}
\label{pt-mult-bis}
\Psi(b) + \int_a^b\|\pt p(t)\|^2\d t
&\leq \Psi(a)-M\int_a^b\int_0^\infty[\dot\mu_t(s)+\mu_t'(s)] \|\psi^t(s)\|^2_1\d s\,\d t\\
\notag
&\quad +2\varpi\int_a^b\| p(t) \|^2_1\d t+ \varpi\|g\|^2(b-a)\\
\notag
&\quad+\frac{1}{\varpi}\int_a^b \Q(\|p(t)\|_1)\kappa(t)\|\psi^t \|^2_{\M_t}\d t.
\end{align}
\end{lemma}

\begin{proof}
Inequality \eqref{p-mult-bis} follows from Lemma \ref{lemma-Phi-inc}
by choosing $\varpi=\theta/2$ and exploiting \eqref{phi-funz2}.
In order to prove \eqref{pt-mult-bis},
we first note that, by \eqref{hp1-f} together with $f(0)=0$,
$$\|\gamma\|\leq \|f(p)\|+\|g\|\leq \Q(\|p\|_1)\|p\|_1+\|g\|.
$$
Hence, making use of {\bf (M5)},
\begin{align*}
&\frac{2}{\kappa(t)}
\int_0^\infty\mu_t(s)\l\psi^t(s),\gamma(t)\r\d s\\
&\quad\leq C\|\gamma(t)\|\int_0^\infty\mu_t(s)
\|\psi^t(s)\|\d s \\
&\quad\leq C\big[\Q(\|p(t)\|_1)\|p(t)\|_1+\|g\|\big]\sqrt{\kappa(t)}\|\psi^t\|_{\M_t}\\
&\quad\leq \varpi\|p(t)\|^2_1+\varpi\|g\|^2
+\frac1\varpi \Q(\|p(t)\|_1)\kappa(t)\|\psi^t\|^2_{\M_t}.
\end{align*}
Substituting the result in Lemma \ref{lemma-Psi} we are done.
\end{proof}

\section{Dissipativity: Proof of Theorem \ref{MAIN-ABS}}
\label{SDIS}

\noindent
Throughout the section, let $\tau\in\R$ be an arbitrarily fixed initial time,
and let $z_\tau\in\H_\tau$ be an arbitrary initial datum for which
$$\EE(\tau,\tau)\leq R,\quad\text{for some given}\,R\geq 0.$$
We preliminarily show that the energy of the system remains bounded.

\begin{proposition}
\label{prop-bound}
For every $t \geq \tau$
we have the estimate
$$\EE(t,\tau)\leq \Q(R).$$
\end{proposition}

\begin{proof}
For $t\geq\tau$, we define the (Lyapunov) functional
$$
\LL(t) = L(t)+\| \eta^t \|^2_{\M_t},
$$
where
$$L(t) = \|u(t)\|^2_1 + \|\pt u(t)\|^2 + 2\l F(u(t)),1\r - 2\l g,u(t)\r.$$
Exploiting the inequalities (the first one being a consequence of \eqref{hp1-f})
\begin{align*}
2\l F(u),1\r &\leq C\big(1+\|u\|_1^4\big),\\
2|\l g,u\r| &\leq \frac\theta2\|u\|^2_1 + C\|g\|^2,
\end{align*}
we deduce from \eqref{phi-funz1} that
\begin{equation}
\label{contr-E-Z}
\theta \EE(t,\tau)-Q_0\leq \LL(t)
\leq \Q(\EE(t,\tau)),
\end{equation}
for some $Q_0\geq 0$. In particular,
$$Q_0=0\qquad\text{if}\qquad g=0\,\,\,\text{and}\,\,\,c_f=0.$$
Testing  equation \eqref{AAA} with $\pt u$ and integrating on $[a,b]$, we get
$$L(b)+2\int_a^b\l\pt u(t),\eta^t\r_{\M_t}\d t = L(a),$$
for every $b> a\geq \tau$.
On the other hand, an application of Lemma \ref{lemma-eta-norm}
with $\sigma=0$ yields
$$
\|\eta^b\|^2_{\M_b}
-\int_a^b\!\!\int_0^\infty\big[\dot\mu_t(s)+\mu'_t(s)\big]\|\eta^t(s)\|^2_{1}\d s\,\d t
\leq \|\eta^a\|^2_{\M_a}
+2\int_a^b\l\pt u(t),\eta^t\r_{\M_t}\d t.
$$
Adding the two inequalities, we are led to
\begin{equation}
\label{lemma-E}
\LL(b)-\int_a^b\!\!\int_0^\infty\big[\dot\mu_t(s)+\mu'_t(s)\big]\|\eta^t(s)\|^2_{1}\d s\,\d t\leq \LL(a).
\end{equation}
In particular, since by {\bf (M4)} the integral in the left-hand side is positive,
$$
\LL(t) \leq \LL(\tau),
$$
and the desired conclusion follows from~\eqref{contr-E-Z}.
\end{proof}

For $\eps\in(0,1]$ to be suitably fixed later, we introduce the further functional
$$
\Lambda(t) = \LL(t) + 2\eps\big[\Phi(t) + 4\Psi(t)\big] + Q_0,
$$
with $Q_0$ as in \eqref{contr-E-Z} and $\Phi,\Psi$ defined in Section \ref{SSAF} by the choice
$(p,\pt p,\psi)=(u,\pt u,\eta)$.
Due to \eqref{contr-Phi-Psi} and \eqref{contr-E-Z}, and by the proposition above,
the inequality
\begin{equation}
\label{contr-E-Lambda}
\frac{\theta}{2}\EE(t,\tau)\leq
\Lambda(t) \leq \Q(\EE(t,\tau))\leq \Q(R)
\end{equation}
is easily seen to hold, provided that $\eps$ is small enough.

\begin{lemma}
\label{lemma-Lambda}
There exist a constant $Q\geq 0$ (independent of $R$)
and $\eps=\eps(R)>0$ small enough such that for every $b>a\geq\tau$
$$
\Lambda(b)+2\eps\int_r^t\Lambda(y)\d y \leq \Lambda(a)+\eps Q(b-a).
$$
If in addition $g=0$ and $c_f=0$, then $Q=0$.
\end{lemma}

\begin{proof} From
Proposition \ref{prop-bound},
$$\Q(\|u(t)\|_1)\leq \Q(R).$$
Hence,
recalling that $\theta\leq 1$ and
using {\bf (M5)}, collecting \eqref{p-mult-bis} and \eqref{pt-mult-bis} with $\varpi=\theta/32$ we infer that
\begin{align*}
&\Phi(b)+4\Psi(b)+\int_a^b\LL(t)\d t+\frac\theta2\int_a^b\EE(t,\tau)\d t\\
&\leq \Phi(a)+4\Psi(a)-4M\int_a^b\int_0^\infty[\dot\mu_t(s)+\mu_t'(s)] \|\eta^t(s)\|^2_1\d s\,\d t\\
&\quad + \Q(R)\int_a^b\kappa(t)\| \eta^t \|^2_{\M_t}\d t  + Q_1(b-a),
\end{align*}
for some $Q_1\geq 0$, which equals zero when $g=0$ and $c_f=0$.
Therefore, taking into account \eqref{lemma-E}, we arrive at
$$
\Lambda(b)+2\eps \int_a^b\Lambda(t)\d t+{\mathfrak I}_1+{\mathfrak I}_2\leq \Lambda(a)+\eps Q(b-a),
$$
where $Q=2(Q_0+Q_1)$ and
\begin{align*}
{\mathfrak I}_1 &=\eps \int_a^b \Big(\theta\EE(t,\tau)-4\eps\big[\Phi(t) + 4\Psi(t)\big]\Big)\d t\\
{\mathfrak I}_2 &=-(1-8\eps M)\int_a^b\!\!\int_0^\infty\big[\dot\mu_t(s)+\mu'_t(s)\big]\|\eta^t(s)\|^2_{1}\d s\,\d t
- \eps \Q(R)\int_a^b\kappa(t)\| \eta^t \|^2_{\M_t}\d t.
\end{align*}
Observe that $Q=0$ if $g=0$ and $c_f=0$.
The proof is finished if we show that ${\mathfrak I}_1$ and ${\mathfrak I}_2$ are nonnegative.
Concerning ${\mathfrak I}_1$, this easily follows from \eqref{contr-Phi-Psi} up to taking $\eps$ sufficiently small.
Coming to ${\mathfrak I}_2$, we first apply {\bf (M4)} to get
$${\mathfrak I}_2\geq \big[\delta(1-8\eps M)- \eps \Q(R)\big]\int_a^b\kappa(t)\| \eta^t \|^2_{\M_t}\d t,$$
up to taking $\eps$ small enough such that $1-8\eps M>0$.
A further reduction of $\eps$, so that
$$\delta(1-8\eps M)\geq \eps \Q(R),$$
yields the claim. Note that the obtained value of $\eps$ depends on $R$.
\end{proof}

\subsection*{Proofs of Theorem \ref{MAIN-ABS} and Corollary \ref{corMAIN-ABS}}
After Lemma~\ref{lemma-Lambda}, we are in a position to
apply the Gronwall-type Lemma \ref{lemma-new-gw} in the Appendix B
with $q_1=0$ and $q_2=\eps Q$, to get
$$\Lambda(t)\leq \Lambda(\tau)\e^{-\eps(t-\tau)} + \frac{\eps Q \e^{\eps}}{1-\e^{-\eps}}.$$
Therefore, by \eqref{contr-E-Lambda},
$$\EE(t,\tau)\leq \frac\theta2\Lambda(t)\leq \Q(R)\e^{-\eps(t-\tau)} +R_0,$$
where
$$R_0=\frac{\theta Q}{2}\sup_{\eps\in(0,1]}\frac{\eps \e^{\eps}}{1-\e^{-\eps}}.
$$
The constant $R_0$, which is independent of $R$, equals zero when $g=0$ and $c_f=0$.
The latter occurs whenever \eqref{hp2bis-f} holds
(see Remark~\ref{remzeroo}). In fact, such a result is already enough in order to
obtain the time-dependent absorbing set. However, by means of a simple trick, it is possible to show
that the decay rate $\eps=\eps(R)$ is actually independent of $R$. This is particularly relevant
in the case $R_0=0$, for it gives the exponential decay of the energy.
To this end, we observe that there is an elapsed time $\tau_\star=\tau_\star(R)\geq 0$ such that
$$\EE(t_\star,\tau)\leq R_0+1 \qquad\text{where}\qquad
t_\star=\tau_\star+\tau.$$
Hence, for $t\geq t_\star$,
$$
\EE(t,\tau)
= \frac12 \| U(t,t_\star)U(t_\star,\tau)z_\tau\|^2_{\H_t}
\leq \Q(R_0+1)\e^{-\eps(t-t_\star)} + R_0,
$$
where $\eps=\eps(R_0+1)$.
On the other hand, if $t\in[\tau,t_\star]$,
$$
\EE(t,\tau)\leq \Q(R).
$$
Then, collecting the two inequalities, for every
$t \geq \tau$ we obtain
$$
\EE(t,\tau)
\leq \Q(R)\e^{-\eps(t-t_\star)} + R_0
= \Q(R)\e^{\eps \tau_\star}\e^{-\eps(t-\tau)}+ R_0
=\Q(R)\e^{-\eps(t-\tau)}+ R_0.
$$
In summary, the desired estimate holds
by setting
$$\omega=\eps(R_0+1),$$
which is now independent of $R$.
\qed

\medskip
We conclude the section by showing the existence of a dissipation integral, which will be needed later.

\begin{lemma}
\label{lemma-diss-ut}
For every $\eps\in(0,1]$ and every $b>a\geq \tau$,
$$
\int_a^b \| \pt u(t)\|^2\d t
\leq \eps(b- a) + \frac{\Q(R)}{\eps^2}.
$$
\end{lemma}

\begin{proof}
For $\varpi\in(0,1]$, we consider the functional
$$
\Upsilon(t) = \L(t) + \varpi^2\Psi(t).
$$
Exploiting \eqref{contr-Phi-Psi}, \eqref{pt-mult-bis}, \eqref{lemma-E} and Proposition~\ref{prop-bound}, we deduce
the integral inequality
$$
\varpi^2\int_a^b\|\pt u(t)\|^2\d t +{\mathfrak I}
\leq \varpi^3\Q(R)(b-a)+\Upsilon(a)-\Upsilon(b)\leq \varpi^3\Q(R)(b-a)+\Q(R),
$$
having set
$${\mathfrak I}=-(1-\varpi^2M)\int_a^b\!\!\int_0^\infty[\dot\mu_t(s)+\mu_t'(s)]\|\eta^t(s)\|^2_1\d s\,\d t
-\varpi \Q(R)\int_a^b \kappa(t)\|\eta^t \|^2_{\M_t}\d t.$$
Owing to {\bf (M4)}, up to possibly reducing $\varpi$,
$${\mathfrak I}\geq \big[\delta(1-\varpi^2 M)-\varpi\Q(R)\big]\int_a^b \kappa(t)\|\eta^t\|^2_{\M_t}\d t\geq 0.
$$
Therefore,
$$
\int_a^b\|\pt u(t)\|^2\d t
\leq \varpi\Q(R)(b-a)+\frac{\Q(R)}{\varpi^2}.
$$
If $\eps>0$ is small enough, up to taking  $\varpi=\eps/\Q(R)$ we obtain the desired inequality,
which then clearly holds for every
$\eps\in(0,1]$.
\end{proof}

\section{The Global Attractor: Proof of Theorem \ref{MAIN-ATTRACTOR}}
\label{SATT}

\noindent
We show that the
process $U(t,\tau):\H_\tau\to\H_t$ possesses
the time-dependent global attractor $\att$.
The proof leans on an abstract result devised in \cite{CPT}, saying
that $\att$ exists if and only if there is
a {\it pullback attracting family}\footnote{The notion
of pullback attracting family is given in Definition~\ref{def-pulla}.} ${\mathfrak C}=\{C_t\}_{t\in\R}$
whose sections $C_t$ are compact in $\H_t$.
Such an attractor $\att$ is invariant as well. Indeed, we know from
\cite[Theorem 5.6]{CPT} that, whenever it exists, the time-dependent global
attractor of a continuous-in-space process is always invariant.
And in our case, in the light of Theorem~\ref{thm-ex-un}, our process $U(t,\tau)$ is
even Lipschitz continuous.

\smallskip
Accordingly, the aim is to find a pullback attracting family with compact sections.
To this end,
let $\mathfrak{B} = \{B_t\}_{t \in \R}$ be a time-dependent absorbing set for the process $U(t,\tau)$.
In what follows, $\tau\in\R$ is an arbitrarily fixed
starting time, and $z_\tau \in B_\tau$ an arbitrarily fixed initial datum.
Throughout this section, the generic constant
$C> 0$ may depend on $\mathfrak{B}$, but is independent of $\tau$ and of the particular choice of $z_\tau$.
By Proposition \ref{prop-bound},
\begin{equation}
\label{bdd-U-B0}
\|U(t,\tau)z_\tau\|_{\H_t}\leq C, \quad\forall t\geq\tau.
\end{equation}
Following a pretty standard procedure (see e.g.\ \cite{LSD}), we split
the nonlinearity $f$ into the sum
$$f(u) = f_0(u) +f_1(u),$$
where
$f_1\in \C^2(\R)$ is globally Lipschitz with $f_1(0) = 0$, while
$f_0 \in \C^2(\R)$ vanishes inside $[-1,1]$ and
fulfills the inequalities
\begin{align}
\label{hp2-phi}
&|f_0''(u)| \leq C|u|,\\
\label{hp1-phi}
&f'_0(u)\geq 0.
\end{align}
Then, we decompose the solution $U(t,\tau)z_\tau$ as
$$U(t,\tau)z_\tau=U_0(t,\tau)z_\tau+U_1(t,\tau)z_\tau,$$
where
$$U_0(t,\tau)z_\tau = (v(t), \pt v(t), \xi^t)\and
U_1(t,\tau)z_\tau = (w(t), \pt w(t), \zeta^t)$$
solve the problems
\begin{equation}
\label{sis-mem-v}
\begin{cases}
\ptt v(t) + Av(t) + \displaystyle{\int_0^\infty \mu_t(s) A\xi^t(s)\d s}
+ f_0(v(t))= 0, \\
\noalign{\vskip1mm}
U_0(\tau,\tau)z_\tau = z_\tau,
\end{cases}
\end{equation}
where
$$\xi^t(s) =
\begin{cases}
v(t) - v(t-s),                     &s \leq t- \tau,\\
\xi_{\tau}(s-t+\tau) + v(t) - v_\tau, &s > t -\tau,
\end{cases}$$
and
\begin{equation}
\label{sis-mem-w}
\begin{cases}
\ptt w(t) + Aw(t) + \displaystyle{\int_0^\infty \mu_t(s) A\zeta^t(s)\d s}
+ f_0(u(t)) - f_0(v(t)) +f_1(u(t))= g, \\
\noalign{\vskip1mm}
U_1(\tau,\tau)z_\tau = 0,
\end{cases}
\end{equation}
where
$$\zeta^t(s) =
\begin{cases}
w(t) - w(t-s),                     &s \leq t- \tau,\\
\zeta_{\tau}(s-t+\tau) + w(t) - w_\tau, &s > t -\tau.
\end{cases}$$
Note that $U_0(t,\tau)$ is itself a process. Besides, due to~\eqref{hp1-phi},
Corollary~\ref{corMAIN-ABS} applies as well
to~\eqref{sis-mem-v}, providing the exponential decay
\begin{equation}
\label{dec-exp-B0}
\|U_0(t,\tau)z_\tau\|^2_{\H_t}\leq
C\e^{-\omega(t-\tau)}.
\end{equation}

\begin{lemma}
\label{bound-wt-1-3}
For every $t\geq\tau$, we have the estimate
$$\EE_1(t)=\frac12\| U_1(t,\tau)z_\tau\|^2_{\H^{1/3}_t}\leq C.
$$
\end{lemma}

\begin{proof}
For $t\geq\tau$, we consider the functional
$$
\LL_1(t) = L_1(t)+\| \zeta^t \|^2_{\M^{1/3}_t},
$$
where
$$L_1(t)=\|w(t)\|^2_{4/3}+\|\pt w(t)\|^2_{1/3}
+ 2\l \gamma(t), A^{1/3}w(t)\r,$$
having defined
$$\gamma(t)=f_0(u(t)) - f_0(v(t)) +f_1(u(t))-g.$$
Collecting \eqref{hp1-f}, \eqref{bdd-U-B0}, \eqref{hp2-phi}
and \eqref{dec-exp-B0}, we learn that
\begin{equation}
\label{bdd-f-phi}
\| \gamma(t)\|\leq C,
\end{equation}
and we readily deduce the controls
\begin{equation}
\label{contr-U1-E1}
\frac32\EE_1(t) - C
\leq \LL_1(t)
\leq \frac52\EE_1(t)+ C.
\end{equation}
A multiplication in $\h^{1/3}$ of the first equation of \eqref{sis-mem-w} by $2\pt w$
gives
$$
\begin{aligned}
&\ddt L_1 + 2\l\zeta,\pt w\r_{\M_t^{1/3}}\\
&\quad=2\l [f_0'(u) - f_0'(v)]\pt u, A^{1/3}w\r
+ 2\l f_0'(v)\pt w, A^{1/3}w\r
+ 2\l f'_1(u)\pt u, A^{1/3}w\r.
\end{aligned}
$$
The estimate of the right-hand side is completely standard, by means of the
(uniform) bounds \eqref{bdd-U-B0}, \eqref{hp2-phi}, \eqref{dec-exp-B0}, $|f'_1|\leq C$,
along with the Sobolev embeddings
$$\h^\sigma \subset L^{6/(3-2\sigma)}(\Omega),\quad \sigma\in \textstyle (0,\frac32).$$
Namely,
\begin{align*}
2\l [f_0'(u) - f_0'(v)]\pt u, A^{1/3}w\r
&\leq C\big(1+\|u\|_{L^6}+\|v\|_{L^6}\big)\|\pt u\|\|w\|_{L^{18}}\|A^{1/3}w\|_{L^{18/5}}\\
&\leq C\|\pt u\|\|w\|^2_{4/3},\\
\noalign{\vskip1mm}
2\l f_0'(v)\pt w, A^{1/3}w\r
&\leq C\|v\|_{L^6}^2\|\pt w\|_{L^{18/7}}\|A^{1/3}w\|_{L^{18/5}}\\
&\leq C\|v\|_1^2\big(\|w\|_{4/3}^2+\|\pt w\|_{1/3}^2\big),\\
\noalign{\vskip1mm}
2\l f'_1(v)\pt u, A^{1/3}w\r
&\leq C\|\pt u \| + C\|\pt u \|\| w \|_{4/3}^2.
\end{align*}
In summary, defining
$$q(t)=\|\pt u(t)\|+\|v(t)\|_1^2,$$
and recalling \eqref{contr-U1-E1},
the right-hand side above is bounded by
$Cq+Cq\L_1$.
Note that, by
Lemma~\ref{lemma-diss-ut} and \eqref{dec-exp-B0},
\begin{equation}
\label{montisola}
\int_a^b \| q(t)\|\d t
\leq \eps(b- a) + \frac{C}{\eps^3},
\end{equation}
for every $b>a\geq \tau$ and every $\eps\in(0,1]$.
Hence,
an integration on $[a,b]$ yields
$$L_1(b)+2\int_a^b\l\pt w(t),\zeta^t\r_{\M_t^{1/3}}\d t
\leq L_1(a)+C\int_a^b q(t) \L_1(t)\d t+C\int_a^b q(t)\d t.$$
At the same time, from Lemma \ref{lemma-eta-norm}
for $\sigma=1/3$,
$$
\|\zeta^b\|^2_{\M_b^{1/3}}
-\int_a^b\!\!\int_0^\infty\big[\dot\mu_t(s)+\mu'_t(s)\big]\|\zeta^t(s)\|^2_{4/3}\d s\,\d t
\leq \|\zeta^a\|^2_{\M_a^{1/3}}
+2\int_a^b\l\pt w(t),\zeta^t\r_{\M_t^{1/3}}\d t.
$$
Adding the two integral inequalities,
we end up with
\begin{align}
\label{ineq-E1}
&\L_1(b)
-\int_a^b\!\!\int_0^\infty\big[\dot\mu_t(s)+\mu'_t(s)\big]\|\zeta^t(s)\|^2_{4/3}\d s\,\d t\\
\notag
&\quad\leq \L_1(a)+ C\int_a^b q(t) \L_1(t)\d t + C\int_a^b q(t)\d t.
\end{align}
For $\eps\in (0,1]$ to be chosen later, we introduce the further functional
$$\Lambda_1(t)= \LL_1(t)
+ 2\eps\big[\Phi(t) +4\Psi(t)\big],$$
where
$\Phi$ and $\Psi$ are defined as in Section \ref{SSAF}, using the triplet
$$(p,\pt p, \psi)=
(A^{1/6}w,A^{1/6}\pt w, A^{1/6}\zeta).$$
In particular, \eqref{contr-Phi-Psi} now reads
\begin{equation}
\label{contr-Phi-Psi-1}
|\Phi(t)|+|\Psi(t)|
\leq C\EE_1(t).
\end{equation}
Consequently,
owing to \eqref{contr-U1-E1},
for $\eps$ sufficiently small
$\Lambda_1$ is controlled by
\begin{equation}
\label{contr-U1-L1}
\EE_1(t) - C
\leq \Lambda_1(t)
\leq 3\EE_1(t) + C.
\end{equation}
Exploiting Lemmas \ref{lemma-Phi-inc} and \ref{lemma-Psi} with $\varpi=1/10$, we infer that
\begin{align*}
&\Phi(b)+4\Psi(b)+\frac32\int_a^b\|w(t)\|^2_{4/3}\d t+2\int_a^b\|\pt w(t)\|^2_{1/3}\d t\\
&\leq \Phi(a)+4\Psi(a)-4M\int_a^b\int_0^\infty[\dot\mu_t(s)+\mu_t'(s)] \|\zeta^t(s)\|^2_{4/3}\d s\,\d t
+ C\int_a^b\kappa(t)\| \zeta^t \|^2_{\M_t^{1/3}}\d t
\\
&\quad-
2\int_a^b \l \gamma(t), A^{1/3} w(t)\r\d t+
\int_a^b\frac{2}{\kappa(t)}\int_0^\infty\mu_t(s)
\l \gamma(t),A^{1/3}\zeta^t(s)\r\d s\d t.
\end{align*}
This inequality, making use of {\bf (M5)}, \eqref{bdd-f-phi}
and \eqref{contr-U1-L1}, enhances to
\begin{align*}
&\Phi(b)+4\Psi(b)+\int_a^b\Lambda_1(t)\d t\\
&\leq \Phi(a)+4\Psi(a)-4M\int_a^b\int_0^\infty[\dot\mu_t(s)+\mu_t'(s)] \|\zeta^t(s)\|^2_{4/3}\d s\,\d t\\
&\quad + C\int_a^b\kappa(t)\| \zeta^t \|^2_{\M_t^{1/3}}\d t
+C(b-a),
\end{align*}
and taking into account \eqref{ineq-E1}-\eqref{contr-U1-L1}
we arrive at
$$
\Lambda_1(b)+2\eps \int_a^b\Lambda_1(t)\d t+
{\mathfrak I}\leq \Lambda_1(a)+C\int_a^b q(t) \Lambda_1(t)\d t + C\int_a^b \big[q(t)+1\big]\d t,
$$
where
\begin{align*}
{\mathfrak I} &=-(1-8\eps M)\int_a^b\!\!\int_0^\infty\big[\dot\mu_t(s)+\mu'_t(s)\big]\|\zeta^t(s)\|^2_{4/3}\d s\,\d t
- \eps C\int_a^b\kappa(t)\| \zeta^t \|^2_{\M_t^{1/3}}\d t.
\end{align*}
Owing to {\bf (M4)}, we get
$${\mathfrak I}\geq \big[\delta(1-8\eps M)- \eps C\big]\int_a^b\kappa(t)\| \zeta^t \|^2_{\M_t^{1/3}}\d t,$$
which is nonnegative up to fixing $\eps$ sufficiently small that
$$\delta(1-8\eps M)\geq \eps C. $$
In conclusion,
$$\Lambda_1(b)-\Lambda_1(a) + 2\eps\int_a^b\Lambda_1(t)\d t
\leq \int_a^b q_1(t) \Lambda_1(t)\d t +\int_a^b q_2(t)\d t, $$
with
$$q_1(t)= C q(t) \and q_2(t)=Cq(t)+C.$$
Keeping in mind
\eqref{montisola},
we can apply the Gronwall-type
Lemma~\ref{lemma-new-gw}. Since $\Lambda_1(\tau)=0$,
this gives
$$\Lambda_1(t) \leq  C, \quad \forall t \geq \tau.$$
The claim then follows from \eqref{contr-U1-L1}.
\end{proof}

\subsection*{Proof of Theorem \ref{MAIN-ATTRACTOR}}
Since $\mathfrak{B} = \{B_t\}_{t \in \R}$ is a time-dependent absorbing set,
collecting~\eqref{dec-exp-B0} and Lemma~\ref{bound-wt-1-3} we infer
that the family of $\H_t^{1/3}$-balls
$$\mathfrak{B}^\star = \big\{\BB_t^{1/3}(r)\big\}_{t \in \R}$$
is pullback attracting provided that $r>0$ is sufficiently large, for
$$\di_{\H_t}\big(U(t,\tau)B_\tau, \BB_t^{1/3}(r)\big)
\leq \sup_{z_\tau\in B_\tau}\|U_0(t,\tau)z_\tau\|_{\H_t}
\leq C\e^{-\frac{\omega}{2}(t-\tau)}.$$
Unfortunately, this is not enough to conclude.
Indeed, although
closed balls of $\H_t^{1/3}$ are uniformly bounded\footnote{Note that,
for any $\sigma>0$,
the embedding constant of $\M_t^{\sigma}\subset \M_t $ is independent of $t$.}, they fail to be compact in $\H_t$,
due to the lack of compactness of the embedding $\M_t^{1/3}\subset \M_t $ (see \cite{PZ}).
On the other hand, up to possibly enlarge $r$,
it is possible to find, for every fixed $t$, a compact set $C_t\subset \BB_t^{1/3}(r)$
such that $\mathfrak{C}=\{C_t\}_{t\in\R}$ is still pullback attracting.
The argument goes word by word as in the proof of Lemma~7.2 in \cite{Visco},
where the same model is considered for a constant-in-time memory kernel,
and is therefore omitted.
Accordingly, the existence of
the time-dependent global attractor
$\att=\{A_t\}_{t\in\R}$
is attained.
\qed

\section{Regularity: Proof of Theorem \ref{regMAIN-ATTRACTOR}}
\label{SREG}

\noindent
In the previous proof, we found a
pullback attracting family $\mathfrak{C}=\{C_t\}_{t\in\R}$
whose sections $C_t$ are compact and uniformly bounded in
$\H_t^{1/3}$. Since, due to
the minimality property of the attractor, the inclusion $A_t \subset C_t$ holds for every $t\in\R$,
it follows that
$$
\sup_{t\in\R}\|A_t\|_{\H_t^{1/3}}<\infty,
$$
meaning that the sections $A_t$ are uniformly bounded in $\H_t^{1/3}$.
The next step is to improve such a regularity.
To this end, we exploit once more the decomposition strategy discussed above, but
for initial data on the attractor. Namely,
for arbitrarily fixed $\tau\in\R$
and
$z_\tau \in A_\tau$,
let us write
$$U(t,\tau)z_\tau=U_0(t,\tau)z_\tau+U_1(t,\tau)z_\tau,$$
where $U_0(t,\tau)$ and $U_1(t,\tau)$ satisfy \eqref{sis-mem-v} and \eqref{sis-mem-w}, respectively,
with
$$f_0(u)=0\and f_1(u)=f(u).$$
In particular, Corollary~\ref{corMAIN-ABS} applies to
$U_0(t,\tau)$, yielding the exponential decay
\begin{equation}
\label{dec-exp-A}
\|U_0(t,\tau)z_\tau\|^2_{\H_t}\leq
C\e^{-\omega(t-\tau)}.
\end{equation}
Here and in what follows, the generic constant $C>0$ depends only on the attractor $\att$
(and is independent of $\tau$ and $z_\tau \in A_\tau$).

\begin{lemma}
\label{lemma-wt-1}
For every $t\geq\tau$, we have the estimate
$$\EE_2(t)=\frac12\|U_1(t,\tau)z_\tau\|_{\H_t^1}^2\leq C.
$$
\end{lemma}

\begin{proof}
We just sketch the proof, which is  completely
analogous to the one of Lemma~\ref{bound-wt-1-3}.
By standard multiplications
and Lemma~\ref{lemma-eta-norm}
with $\sigma=1$, the functional
$$\LL_2(t) = \|w(t)\|^2_2 + \|\pt w(t)\|^2_1
+\| \zeta^t \|^2_{\M^{1}_t}
+ 2\l f(u(t)), Aw(t)\r$$
is shown to fulfill the integral inequality
\begin{align}
\label{ineq-E2-b}
&\L_2(b)
-\int_a^b\!\!\int_0^\infty\big[\dot\mu_t(s)+\mu'_t(s)\big]\|\zeta^t(s)\|^2_{2}\d s\,\d t\\ \nonumber
&\quad\leq \L_2(a)+\frac{\eps}{2}\int_a^b\|w(t)\|_2^2\d t +\frac{C}{\eps}(b-a)
\end{align}
for every $b>a\geq \tau$ and every $\eps\in(0,1]$. Here, the only difference with
respect to the previous proof is the control
$$
2\l f'(u)\pt u, Aw\r\leq
C\big(1+\|u\|_{L^{18}}\big)\|\pt u\|_{L^{18/7}}\|Aw\|
\leq \frac{\eps}{2}\|w\|_2^2 +\frac{C}{\eps}.
$$
Then we introduce the functional
$$\Lambda_2(t)
= \L_2(t)
+ 2\eps[\Phi(t) +
4\Psi(t)],$$
where
$\Phi$ and $\Psi$ are defined as in Section \ref{SSAF}, using the triplet
$$(p,\pt p,\psi)=
(A^{1/2}w,A^{1/2}\pt w, A^{1/2}\zeta).$$
For $\eps$ sufficiently small,
\begin{equation}
\label{contr-U2-L2}
\EE_2(t) - C
\leq \Lambda_2(t)
\leq 3\EE_2(t) + C.
\end{equation}
Besides, exploiting Lemma \ref{lemma-Phi-inc} and
Lemma \ref{lemma-Psi} with $\gamma = f(u)$ and
$\varpi = 1/20$, by standard computations we obtain
\begin{align*}
&\Phi(b)+4\Psi(b)+\frac74\int_a^b\|w(t)\|^2_{2}\d t+2\int_a^b\|\pt w(t)\|^2_{1}\d t\\
&\leq \Phi(a)+4\Psi(a)
-4M\int_a^b\int_0^\infty[\dot\mu_t(s)+\mu_t'(s)] \|\zeta^t(s)\|^2_{2}\d s\,\d t\\
&\quad + C\int_a^b\kappa(t)\| \zeta^t \|^2_{\M_t^{1}}\d t
+ C(b-a).
\end{align*}
Adding this inequality with \eqref{ineq-E2-b}, and taking into account \eqref{contr-U2-L2},
we arrive at
$$
\Lambda_2(b)+2\eps \int_a^b\Lambda_2(t)\d t \leq \Lambda_2(a)+ \frac{C}{\eps}(b-a),
$$
up to fixing $\eps$ sufficiently small.
Hence, the Gronwall-type Lemma~\ref{lemma-new-gw} with $q_1= 0$ and $q_2= C$
applies. Since $\Lambda_2(\tau)=0$, this
gives (now $\eps$ is fixed)
$$\Lambda_2(t) \leq  C, \quad \forall t \geq \tau,$$
and a further exploitation of \eqref{contr-U2-L2} completes the argument.
\end{proof}

\subsection*{Proof of Theorem \ref{regMAIN-ATTRACTOR}}
Up to taking $r>0$ sufficiently large, Lemma \ref{lemma-wt-1} tells that
$$U_1(t,\tau)A_\tau \subset \BB_t^1(r), \quad\forall t\geq\tau.$$
Thus, for every $t\geq\tau$, we infer from
the invariance of $\att$ and~\eqref{dec-exp-A} that
$$\di_{\H_t}\big(A_t, \BB_t^1(r)\big)
=\di_{\H_t}\big(U(t,\tau)A_\tau, \BB_t^1(r)\big)
\leq C\e^{-\frac{\omega}{2}(t-\tau)}.$$
Since the inequality holds for every $\tau\in\R$,
letting $\tau\to-\infty$ we reach the conclusion
$$\di_{\H_t}\big(A_t, \BB_t^1(r)\big)=0
\quad\Rightarrow\quad
A_t\subset \BB_t^1(r).$$
Accordingly, the $\H^1_t$-norm of $A_t$ is bounded by $r$ for every $t\in\R$.
\qed

\section{Recovering Kelvin-Voigt: Proof of Theorem \ref{limitece}}
\label{SKV}

\noindent
By Remark \ref{NEWEQ}, for every $n$ the function
$w_n(\cdot)=u_n(\cdot+t_n)$
fulfills the equation
\begin{equation}
\label{mff2n}
\ptt w_n(t)+A w_n(t) +\int_0^\infty k_{t+t_n}(s)A\pt w_n(t-s)\d s + f(w_n(t)) = g.
\end{equation}
Besides, from the estimates of Theorem \ref{regMAIN-ATTRACTOR},
\begin{equation}
\label{u-Linfinito}
w_n \,\,\,\text{is bounded in }\, L^\infty(\R;\h^2)\cap W^{1,\infty}(\R;\h^1).
\end{equation}
Therefore, there exists
$$\hat u\in L^\infty(\R;\h^2)\cap W^{1,\infty}(\R;\h^1)$$
such that, up to a subsequence,
\begin{align*}
w_n \rightharpoonup \hat u & \quad\text{ weakly* in } L^\infty(\R;\h^2),\\
\pt w_n \rightharpoonup \pt \hat u & \quad\text{ weakly* in } L^\infty(\R;\h^1).
\end{align*}
Since
$$\lim_{n\to\infty} k_{t+t_n}=m\delta_0$$
in the sense of \eqref{deltaconv}, we deduce that, for every $T>0$,
$$\sup_{t\in [-T,T]}\int_0^\infty k_{t+t_n}(s)\d s \leq 2m,$$
for every $n$ sufficiently large (depending on $T$).
Hence, exploiting \eqref{hp1-f} and the uniform estimate \eqref{u-Linfinito}, by comparison in \eqref{mff2n} we
obtain, for all $n$ large,
$$\sup_{t\in [-T,T]}\,\|\ptt w_n(t)\|_{-1}\leq C,$$
for some $C>0$ (which is actually independent of $T$), and we conclude that
$$\pt w_n\,\,\,\text{is bounded in }\,L^\infty(-T,T;\h^1)\cap W^{1,\infty}(-T,T;\h^{-1}).$$
By the classical Simon-Aubin Theorem~\cite{SIM}, we have the compact embeddings
\begin{align*}
L^\infty(-T,T;\h^2)\cap W^{1,\infty}(-T,T;\h^1)&\Subset\C([-T,T],\h^1),\\
L^\infty(-T,T;\h^1)\cap W^{1,\infty}(-T,T;\h^{-1})&\Subset \C([-T,T],\h).
\end{align*}
Accordingly, the strong convergence
$$(w_n,\pt w_n)\to (\hat u,\pt \hat u)\quad\text{in } \C([-T,T],\h^1\times \h)$$
holds (up to a subsequence), implying in particular~\eqref{conv-cc}.
We are left to show that $(\hat u,\partial_t \hat u)$ solves the strongly damped wave equation, namely
$$
\ptt \hat u(t)+A \hat u(t) +A\pt \hat u(t) + f(\hat u(t)) = g.
$$
Indeed, we are going to prove that the equality above is recovered when passing to the limit as $n\to\infty$ in \eqref{mff2n},
the only nonstandard convergence being
$$\int_0^\infty k_{t+t_n}(s)A\pt w_n(t-s)\d s \to m A \pt \hat u(t).$$
This follows if we can show that, for an arbitrarily fixed $t$,
$$\mathfrak I_{n}(t)=\int_0^\infty k_{t+t_n}(s)\l \pt w_n(t-s),A\phi\r \d s \to m \l \pt \hat u(t),A\phi\r,$$
for any sufficiently regular $\phi$ (say, $\phi\in\h^2$).
To this end, let $\phi$ be fixed, and call for simplicity
$$\varphi_n(\cdot)= \l \pt w_n(\cdot),A\phi\r
\and
\varphi(\cdot)=\l \pt \hat u(\cdot),A\phi\r.$$
Then, we write
$$\mathfrak I_{n}(t)=\mathfrak I_{n}^1(t)+\mathfrak I_{n}^2(t)+\mathfrak I_{n}^3(t),$$
where
\begin{align*}
{\mathfrak I}_{n}^1(t)&=\int_0^1 k_{t+t_n}(s)\varphi(t-s)\d s,\\
{\mathfrak I}_{n}^2(t)&=\int_0^1 k_{t+t_n}(s)\big[\varphi_n(t-s)-\varphi(t-s)\big]\d s,\\
{\mathfrak I}_{n}^3(t)&=\int_1^\infty k_{t+t_n}(s)\varphi_n(t-s)\d s.
\end{align*}
Since $\varphi_n\in L^\infty(\R)$ uniformly with respect to $n$,
and $ \varphi_n\to \varphi$ in $\C(I)$ for every closed interval $I$, it is apparent from~\eqref{deltaconv} that
$$\big|{\mathfrak I}_{n}^2(t)+{\mathfrak I}_{n}^3(t)\big|
\leq \|\varphi_n-\varphi\|_{\C([t-1,t])}\int_0^1 k_{t+t_n}(s)\d s
+ \|\varphi_n\|_{L^\infty(\R)}\int_1^\infty k_{t+t_n}(s)\d s\to 0,
$$
while
$${\mathfrak I}_{n}^1(t)\to m\varphi(t),
$$
proving the required convergence.
\qed

\section{Two Memory Kernels of Physical Interest}
\label{STMK}

\noindent
In this final section we discuss two concrete examples of time-dependent memory kernels arising in the
physical applications, already introduced in~\cite{Timex}.

\subsection*{I. The rescaled kernel}
Let $\mu\in \C^1(\R^+)\cap L^1(\R^+)$ be a (nonnull and nonnegative) nonincreasing
function with $\mu(0)<\infty$.
Given a bounded positive function $\eps\in\C^1(\R)$ satisfying
$$
\dot\eps(t)\leq 0, \quad\forall t\in\R,
$$
we define the time-dependent rescaled kernel
$$\mu_t(s)
=\frac{1}{[\eps(t)]^2}\,\mu\bigg(\frac{s}{\eps(t)}\bigg).$$
According to \eqref{kkk}, the corresponding integrated memory kernel reads
$$k_t(s)
=\frac{1}{\eps(t)}\,k\bigg(\frac{s}{\eps(t)}\bigg)\qquad\text{where}\qquad
k(s)=\int_s^\infty \mu(y)\d y.
$$
In particular, assuming $k$ summable with total mass $m>0$, the most interesting situation
is when $\eps(t)\to 0$ as $t\to\infty$. In which case, we recover the
distributional convergence $k_t\to m\delta_0$ to (a multiple of) the Dirac mass at zero.
As shown in \cite{Timex}, this $\mu_t$ complies with {\bf (M1)}-{\bf (M3)}.
Here, we make two further assumptions: there exists $\varrho>0$ such that
\begin{equation}
\label{mu-delta}
\mu'(s) + \varrho\mu(s) \leq 0,\quad \forall s>0,
\end{equation}
and
\begin{equation}
\label{inf-der-eps}
\inf_{t \in \R} \dot\eps(t) > -\frac{\varrho}{2}.
\end{equation}

\begin{example}
For instance, a possible choice is the exponential kernel $\mu(s)=\e^{-s}$ and
$$\eps(t)=c\Big[\frac{\pi}{2}-\arctan(t)\Big],\quad
0<c<\frac12.$$
\end{example}

Our aim is showing that, under \eqref{mu-delta}-\eqref{inf-der-eps},
the rescaled kernel $\mu_t$ satisfies {\bf (M4)}-{\bf (M8)} as well.
To this end, we recall that
the derivatives of $\mu_t$
read
$$
\mu'_t(s)
=\frac{1}{[\eps(t)]^3}\mu'\bigg(\frac{s}{\eps(t)}\bigg)
\and
\dot\mu_t(s)
=-\frac{\dot\eps(t)}{\eps(t)}\,\big[2\mu_t(s)+s\mu'_t(s)\big].
$$
Besides,
$$\kappa(t)=
\int_0^\infty\mu_t(s)\d s
=\frac{\kappa}{\eps(t)},$$
being $\kappa=\int_0^\infty\mu(s)\d s$ the total mass of $\mu$.

\medskip
\noindent
{\bf -} {\it Verifying} {\bf (M4)}. On account of \eqref{inf-der-eps}, choose $\delta>0$
small enough that
$$\delta\kappa-2\dot\eps(t)\leq \varrho.$$
Then, for every $t\in\R$ and $s>0$,
\begin{align*}
\dot\mu_t(s)+\mu_t'(s)+\delta\kappa(t)\mu_t(s)
&=\bigg[1-\frac{s\dot\eps(t)}{\eps(t)}\bigg]\mu_t'(s)
+\big[\delta\kappa-2\dot\eps(t)\big]\frac{\mu_t(s)}{\eps(t)}\\
\noalign{\vskip1mm}
&\leq \mu_t'(s)+\frac\varrho{\eps(t)}\mu_t(s)\\
&=\frac{1}{[\eps(t)]^3}\bigg[\mu'\bigg(\frac{s}{\eps(t)}\bigg)
+\varrho\mu\bigg(\frac{s}{\eps(t)}\bigg)\bigg]\leq 0.
\end{align*}
The latter inequality follows from \eqref{mu-delta}.

\medskip
\noindent
{\bf -} {\it Verifying} {\bf (M5)}. Simply observe that
$$\inf_{t\in\R}\kappa(t)=\lim_{t\to-\infty}
\frac{\kappa}{\eps(t)}>0.
$$

\medskip
\noindent
{\bf -} {\it Verifying} {\bf (M6)}. In the light of {\bf (M1)},
an integration by parts readily gives
$$\int_0^\infty s\mu'_t(s)\d s =-\kappa(t).$$
Hence, recalling that $\dot\eps\leq 0$ and $\mu_t'\leq 0$,
$$
\int_0^\infty|\dot\mu_t(s)|\,\d s
=-\frac{\dot\eps(t)}{\eps(t)}
\int_0^\infty|2\mu_t(s)+s\mu'_t(s)|\,\d s
\leq-\frac{3\kappa(t)\dot\eps(t)}{\eps(t)}
=-\frac{3\kappa \dot\eps(t)}{[\eps(t)]^2},
$$
and exploiting \eqref{inf-der-eps},
$$\sup_{t\in\R}\frac1{[\kappa(t)]^2}\int_0^\infty
|\dot\mu_t(s)|\,\d s
\leq -\inf_{t\in\R}\frac{3\dot\eps(t)}{\kappa}
<\frac{3\varrho}{2\kappa}.$$

\medskip
\noindent
{\bf -} {\it Verifying} {\bf (M7)}.
By direct calculations,
$$\frac{\mu_t(0)}{[\kappa(t)]^2}
=\frac{\mu(0)}{\kappa^2},\quad\forall t\in\R.$$

\medskip
\noindent
{\bf -} {\it Verifying} {\bf (M8)}.
For $t\in[a,b]$ and $\nu>0$, we have
$$\int_\nu^{1/\nu}\mu_t(s)\d s=\bigg(\frac{1}{\kappa}
\int_{\nu/\eps(t)}^{1/\nu\eps(t)}\mu(s)\d s\bigg)\kappa(t)
\geq \bigg(\frac{1}{\kappa}
\int_{\nu/\eps(a)}^{1/\nu\eps(b)}\mu(s)\d s\bigg)\kappa(t).$$
When $\nu\to 0$, the right-hand side converges to $\kappa(t)$.

\subsection*{II. The rheological kernel}
The second example of time-dependent memory kernel comes from a physical derivation of
equation \eqref{mff} via a rheological model for aging materials (see \cite[Appendix]{Timex} for details).
It has the form
$$\mu_t(s)=\frac{1}{\varrho\gamma}
\Ks_0(t)\Ks_0(t-s)\e^{-\frac1\gamma\int_0^s\Ks_0(t-y)\d y},$$
where $\varrho,\gamma$ are positive constants and  $\Ks_0\in\C^1(\R)$
is a nondecreasing function such that
\begin{equation}
\label{inf-K0}
\lim_{t\to -\infty}\Ks_0(t)= \beta>0.
\end{equation}
The corresponding integrated memory kernel given by \eqref{kkk} reads
$$k_t(s)=\frac{1}{\varrho}
\Ks_0(t)\e^{-\frac1\gamma\int_0^s\Ks_0(t-y)\d y}.$$
Here, the interesting situation is when $\Ks_0(t)\to\infty$ as $t\to\infty$,
translating the physical assumption that the spring in the Maxwell element of the rheological model
becomes completely rigid in the longtime. In which case,
quite remarkably, we have the distributional convergence (see \cite{Timex})
$$k_t\to\frac\gamma\varrho\delta_0 \qquad\text{as }\,t\to\infty.$$
In the same paper, the rheological kernel $\mu_t$ is shown to satisfy
{\bf (M1)}-{\bf (M3)}. In order to verify the remaining axioms {\bf (M4)}-{\bf (M8)},
a further assumptions is needed: there exists a positive constant
$\Mr<\beta/\gamma$ such that
\begin{equation}
\label{K0-subexp}
\dot\Ks_0(t)\leq \Mr \Ks_0(t),\quad \forall t\in\R.
\end{equation}
Loosely speaking, \eqref{K0-subexp} prevents $\Ks_0$ to grow ``too fast".
On the other hand, $\Ks_0$ can even be an exponential.

\begin{example}
The exponential function
$$\Ks_0(t)=1+\e^{\alpha t}$$
fulfills our hypotheses, provided that $0<\alpha<1/\gamma$.
\end{example}

First, let us write explicitly the derivatives of $\mu_t$. Namely,
$$ \mu_t'(s)
=-\frac{\Ks_0(t)}{\varrho\gamma}\bigg[\dot{\Ks}_0(t-s)+
\frac1\gamma[\Ks_0(t-s)]^2\bigg]\e^{-\frac1\gamma\int_0^s \Ks_0(t-y)\d y}$$
and
$$\begin{aligned}
\dot\mu_t(s)&=
\frac{1}{\varrho\gamma}\bigg[\dot{\Ks}_0(t)\Ks_0(t-s)+ \Ks_0(t)\dot{\Ks}_0(t-s)\\
\noalign{\vskip1mm}
&\qquad\quad- \frac{1}{\gamma}[\Ks_0(t)]^2 \Ks_0(t-s)
+\frac{1}{\gamma}\Ks_0(t)[\Ks_0(t-s)]^2\bigg]\e^{-\frac1\gamma\int_0^s \Ks_0(t-y)\d y}.
\end{aligned}$$
Besides,
$$\kappa(t)=\int_0^\infty\mu_t(s)\d s = \frac{\textsf{K}_0(t)}{\varrho}.$$

\medskip
\noindent
{\bf -} {\it Verifying} {\bf (M4)}.
By virtue of \eqref{K0-subexp}, let $0<\delta\leq\varrho/\gamma$ to be properly chosen.
By explicit calculations,
$$\dot\mu_t(s)+\mu_t'(s)+\delta\kappa(t)\mu_t(s)
=\frac{\Ks_0(t-s)}{\varrho^2\gamma^2}\,{\textsf{F}}(t)
\e^{-\frac1\gamma\int_0^s \Ks_0(t-r)\d r},
$$
where
$${\textsf{F}}(t)=\varrho\gamma\dot\Ks_0(t)+(\delta\gamma-\varrho)[ \Ks_0(t)]^2.
$$
The claim amounts to showing that ${\textsf{F}}\leq 0$. Indeed, since
$\Ks_0\geq \beta$ and $\delta\gamma-\varrho\leq 0$, we infer from \eqref{K0-subexp} that
$${\textsf{F}}(t)\leq \big[\varrho\gamma \Mr+(\delta\gamma-\varrho)\beta\big]\Ks_0(t)\leq 0,
$$
as long as we fix $\delta>0$ small enough that
$$\delta\leq \frac{\varrho\big(\beta-\gamma\Mr\big)}
{\gamma\beta}.$$
This is possible because $\beta-\gamma\Mr>0$.

\medskip
\noindent
{\bf -} {\it Verifying} {\bf (M5)}. Due to \eqref{inf-K0} and the fact that $\dot\Ks_0\geq 0$,
$$\inf_{t\in\R}\kappa(t)
=\frac{1}{\varrho}\inf_{t\in\R}\Ks_0(t)=\frac{\beta}{\varrho}>0.$$

\medskip
\noindent
{\bf -} {\it Verifying} {\bf (M6)}.
Since $\dot\Ks_0\geq 0$ and \eqref{inf-K0}-\eqref{K0-subexp} hold,
$$|\dot\mu_t(s)|\leq\frac{2[\Ks_0(t)]^2}{\varrho\gamma}
\bigg(\frac{\Mr}{\beta}+\frac1\gamma\bigg)\Ks_0(t-s)\e^{-\frac1\gamma\int_0^s \Ks_0(t-r)\d r}.
$$
Moreover, as the positive function $\Ks_0$ remains away from zero,
$$\int_0^\infty \Ks_0(t-s)\e^{-\frac1\gamma\int_0^s \Ks_0(t-r)\d r}\d s=
\gamma-\gamma\lim_{s\to\infty}\e^{-\frac1\gamma\int_0^s \Ks_0(t-r)\d r}\d s=\gamma.
$$
Accordingly,
$$\int_0^\infty |\dot\mu_t(s)|\d s\leq\frac{2[\Ks_0(t)]^2}{\varrho}\bigg(\frac{\Mr}{\beta}+\frac1\gamma\bigg).
$$
Hence,
$$\sup_{t\in\R}\frac1{[\kappa(t)]^2}\int_0^\infty |\dot\mu_t(s)|\,\d s \leq 2\varrho\bigg(\frac{\Mr}{\beta}+\frac1\gamma\bigg).
$$

\medskip
\noindent
{\bf -} {\it Verifying} {\bf (M7)}.
By the very definition of $\mu_t$,
$$\frac{\mu_t(0)}{[\kappa(t)]^2}=\frac{\varrho}{\gamma},\quad\forall t\in\R.$$

\medskip
\noindent
{\bf -} {\it Verifying} {\bf (M8)}.
Recalling the form of $k_t(s)$, thanks to \eqref{inf-K0} we easily get
$$
\int_\nu^{1/\nu}\mu_t(s)\d s=\Big[\e^{-\frac1\gamma\int_0^\nu\Ks_0(t-y)\d y}-
\e^{-\frac1\gamma\int_0^{1/\nu}\Ks_0(t-y)\d y}\Big]\frac{\Ks_0(t)}{\varrho}\,\to \kappa(t)
$$
as $\nu\to 0$, uniformly as $t\in[a,b]$.

\section*{Appendix A\\ Proof of Lemma \ref{lemma-Phi-inc}}

\theoremstyle{plain}
\newtheorem{theoremAPPA}{Theorem}[section]
\renewcommand{\thetheoremAPPA}{A.\arabic{theoremAPPA}}
\setcounter{equation}{0}
\setcounter{subsection}{0}
\renewcommand{\theequation}{A.\arabic{equation}}

\noindent
Let $b>a\geq \tau$ be arbitrarily fixed.
In what follows $(p,\pt p,\psi)$ will be a
solution to problem \eqref{sis-p-psi}-\eqref{in-cond-p},
as regular as needed. In particular,
we assume $\psi_\tau\in \D(\T_\tau)\cap \C^1(\R^+,\h^1)$
and $p\in W^{2,\infty}(a,b; \h^{1})$.
In addition, the derivative $\psi_\tau'$ of the initial datum $\psi_\tau$ is required
to fulfill $\psi'_\tau\in L^\infty(0,s_0;\h)$ for every $s_0>0$.
In fact, since we are always working in a Galerkin regularization scheme, this is no loss of generality.
According to~\cite[Section 5]{Timex},
we can differentiate \eqref{rep-p-psi} with respect to $s$ in the weak sense,
so obtaining
\begin{equation}
\label{DERPSI}
\ps\psi^t(s) =
\begin{cases}
\pt p(t-s),                  &s \leq t-\tau,\\
\psi'_\tau(s-t+\tau),        &s > t-\tau.
\end{cases}
\end{equation}
Besides, $\psi^t$ belongs to $\D(\T_\tau)$ for every $t\geq\tau$
and satisfies the differential equality in $\M_\tau$
\begin{equation}
\label{mem-p-psi}
\pt\psi^t = \T_\tau\psi^t + \pt p(t).
\end{equation}

The proof of Lemma \ref{lemma-Phi-inc} will be carried out in a number of steps.
The generic constant $C>0$ appearing in the next lines is understood to be independent of
the interval $[a,b]$.
Instead, we will denote by $D>0$ a generic constant depending explicitly on $[a,b]$.

\subsection*{I. The approximating functional}
We  introduce a family of approximate memory kernels
with compact support in $\R^+$.
For $\eps>0$ small, we define the cut-off function
$$\phi_\eps(s)=
\begin{cases}
0 & \text{if } 0\leq s<\eps,\\
s/\eps-1 & \text{if } \eps\leq s<2\eps,\\
1 & \text{if } 2\eps\leq s\leq 1/\eps,\\
2-\eps s & \text{if }  1/\eps\leq s<2/\eps,\\
0  & \text{if } 2/\eps\leq s,
\end{cases}
$$
along with the approximate kernel
$$\mu_t^\eps(s)=\phi_\eps(s)\mu_t(s)\leq \mu_t(s).$$
Setting
$$\kappa_\eps(t)=\int_0^\infty\mu_t^\eps(s)\d s\leq \kappa(t),$$
we introduce the functional
$$\Psi_\eps(t)= -\frac{2}{\kappa_\eps(t)}
\int_0^\infty\mu_t^\eps(s)\l\psi^t(s),\pt p(t)\r\d s,$$
and we define
$$R_\eps(t)=\bigg[\frac{\kappa(t)}{\kappa_\eps(t)}\bigg]^2\geq 1.$$
Owing to {\bf (M8)}, for all $\eps$ sufficiently small (depending on
$a$ and $b$), we have
\begin{equation}
\label{PRIMOR}
\sup_{t\in [a,b]}R_\eps(t)\leq 4.
\end{equation}
The time-derivative of $\Psi_\eps$ is given by
\begin{equation}
\label{PRIMO}
\ddt\Psi_\eps(t)={\mathfrak I}_1(t)+{\mathfrak I}_2(t),
\end{equation}
where\footnote{We omit for simplicity of notation the dependence on $t$
of the variables $p$ and $\psi$.}
\begin{align*}
{\mathfrak I}_1(t)=& -2\ddt\bigg[\frac{1}{\kappa_\eps(t)}\bigg]
\int_0^\infty\mu_t^\eps(s)\l \psi(s),\pt p\r\d s,\\
{\mathfrak I}_2(t)=&-\frac{2}{\kappa_\eps(t)}
\ddt\int_0^\infty\mu_t^\eps(s)\l \psi(s),\pt p\r\d s.
\end{align*}

\subsection*{II. Estimating $\boldsymbol{{\mathfrak I}_1}$}
We first observe that
$$\ddt\bigg[\frac{1}{\kappa_\eps(t)}\bigg]
=-\frac{1}{[\kappa_\eps(t)]^2}\ddt \int_0^\infty\mu_t^\eps(s)
=-\frac{1}{[\kappa_\eps(t)]^2}\int_0^\infty\dot\mu_t^\eps(s).$$
Passing the derivative through the integral sign is allowed by a classical result. Indeed,
we know from {\bf (M1)} and {\bf (M3)} that
\begin{itemize}
\item[-] the map $s\mapsto \mu_t^\eps(s)$ is summable on $\R^+$ for every fixed $t$;
\item[-] for almost every fixed $s>0$, the map $t\mapsto \mu_t^\eps(s)$
is differentiable for all $t\in\R$;
\item[-] for every fixed closed interval $I\subset \R$ and almost every $t\in I$,
$$|\dot \mu_t^\eps(s)| \leq \|\dot \mu\|_{L^\infty(I\times [\eps,1/\eps])}\,\chi_{[\eps,1/\eps]}(s)\in L^1(\R^+).$$
\end{itemize}
Accordingly, thanks to {\bf (M6)} and \eqref{PRIMOR},
\begin{align}
\label{est1-Psi}
{\mathfrak I}_1(t)
&\leq 2R_\eps(t)\|\pt p\|
\frac1{[\kappa(t)]^2}\int_0^\infty|\dot\mu_t(s)|\d s
\int_0^\infty\mu_t(s)\|\psi(s)\|\d s \\
\notag
&\leq \frac12R_\eps(t)\|\pt p\|^2 +
C\kappa(t)\|\psi\|_{\M_t}^2.
\end{align}

\subsection*{III. Estimating $\boldsymbol{{\mathfrak I}_2}$}
Passing again the derivative within the integral (arguing as
before), we write
$${\mathfrak I}_2(t)={\mathfrak F}(t)-2\|\pt p\|^2
-\frac{2\sqrt{R_\eps(t)}\,}{\kappa(t)}
\int_0^\infty\mu_t^\eps(s)\l\psi(s),\ptt p\r \d s,$$
where we set
$$
{\mathfrak F}(t)=2\|\pt p\|^2-\frac{2}{\kappa_\eps(t)}
\bigg[\int_0^\infty\dot\mu_t^\eps(s)\l\psi(s),\pt p\r\d s
+\int_0^\infty\mu_t^\eps(s)\l\pt\psi(s),\pt p\r \d s\bigg].
$$
We preliminarily observe that, exploiting \eqref{mem-p-psi},
\begin{align*}
\int_0^\infty\mu_t^\eps(s)
\l\pt\psi(s),\pt p\r\d s
&=
-\int_0^\infty\mu_t^\eps(s)\l\ps\psi(s),\pt p\r\d s +\kappa_\eps(t)\|\pt p\|^2\\
&=\int_0^\infty(\mu^\eps_t)'(s)\l\psi(s),\pt p\r\d s+\kappa_\eps(t)\|\pt p\|^2.
\end{align*}
The latter equality follows from an integration by parts, where the boundary terms
are easily seen to vanish. Hence,
$$
{\mathfrak F}(t)
=-\frac{2}{\kappa_\eps(t)}\int_0^\infty[\dot\mu_t^\eps(s)+(\mu_t^\eps)'(s)]\l\psi(s),\pt p\r\d s.
$$
At this point, we further decompose ${\mathfrak F}$ into the sum
$${\mathfrak F}(t)={\mathfrak F}_1(t)+{\mathfrak F}_2(t),$$
where
\begin{align*}
{\mathfrak F}_1(t)=& -\frac{2}{\kappa_\eps(t)}
\int_0^\infty\phi_\eps(s)[\dot\mu_t(s)+\mu_t'(s)]\l\psi(s),\pt p\r\d s,\\
{\mathfrak F}_2(t)=&-\frac{2}{\kappa_\eps(t)}
\int_0^\infty\phi_\eps'(s)
\mu_t(s)\l\psi(s),\pt p\r\d s.
\end{align*}
We estimate these two terms separately.

\smallskip
\noindent
{\bf -}
Since $\dot\mu + \mu'\leq0$,
\begin{align*}
{\mathfrak F}_1(t)
&\leq
2\sqrt{R_\eps(t)}\,\|\pt p\|\,\frac{1}{\kappa(t)}
\int_0^\infty -[\dot\mu_t(s)+\mu_t'(s)]\|\psi(s)\|\d s \\
&\leq \frac12R_\eps(t)\|\pt p\|^2
+\frac{C}{[\kappa(t)]^2}
\bigg[\int_0^\infty -[\dot\mu_t(s)+\mu_t'(s)]\|\psi(s)\|_1\d s\bigg]^2.
\end{align*}
Moreover, owing to {\bf (M6)}-{\bf (M7)},
\begin{align*}
&\frac{C}{[\kappa(t)]^2}\bigg[\int_0^\infty-[\dot\mu_t(s)+\mu_t'(s)]\|\psi(s)\|_1\d s\bigg]^2\\
&\quad\leq
\frac{C}{[\kappa(t)]^2}\int_0^\infty -[\dot\mu_t(s)+\mu_t'(s)]\d s
\int_0^\infty -[\dot\mu_t(s)+\mu_t'(s)]\|\psi(s)\|^2_1\d s\\
&\quad\leq
C\bigg[\frac1{{[\kappa(t)]^2}}\int_0^\infty|\dot\mu_t(s)|\d s+
\frac{\mu_t(0)}{[\kappa(t)]^2}\bigg]
\int_0^\infty -[\dot\mu_t(s)+\mu_t'(s)]\|\psi(s)\|^2_1\d s\\
&\quad\leq - M
\int_0^\infty[\dot\mu_t(s)+\mu_t'(s)]\|\psi(s)\|^2_1\d s,
\end{align*}
for some $M>0$, independent of $[a,b]$. Therefore, we end up with the estimate
$${\mathfrak F}_1(t)\leq
\frac12R_\eps(t)\|\pt p\|^2- M
\int_0^\infty[\dot\mu_t(s)+\mu_t'(s)]\|\psi(s)\|^2_1\d s.$$

\smallskip
\noindent
{\bf -}
Since
$$|\phi_\eps'(s)|\leq \frac{1}{\eps}\chi_{[\eps,2\eps]}(s)+\eps,$$
making use of {\bf (M5)}, \eqref{PRIMOR} and the control $\pt p\in L^\infty(a,b;\h)$, we get
$$
{\mathfrak F}_2(t)\leq D\int_0^\infty r_\eps(t,s)\d s,\quad\forall t\in[a,b],$$
having set
\begin{equation}
\label{ERREEPS}
r_\eps(t,s)=
\bigg(\frac{1}{\eps}\chi_{[\eps,2\eps]}(s)+\eps\bigg)\mu_t(s)\|\psi^t(s)\|.
\end{equation}

\smallskip
\noindent
Collecting the two inequalities, we conclude that
$${\mathfrak F}(t)\leq \frac12R_\eps(t)\|\pt p\|^2- M
\int_0^\infty[\dot\mu_t(s)+\mu_t'(s)]\|\psi(s)\|^2_1\d s
+D\int_0^\infty r_\eps(t,s)\d s.$$
In turn, this yields
\begin{align}
\label{est2-Psi}
{\mathfrak I}_2(t)
&\leq -\bigg(2-\frac12R_\eps(t)\bigg)\|\pt p\|^2
- M\int_0^\infty[\dot\mu_t(s)+\mu_t'(s)]\|\psi(s)\|^2_1\d s\\
\notag
&\quad +D\int_0^\infty r_\eps(t,s)\d s
-\frac{2\sqrt{R_\eps(t)}\,}{\kappa(t)}
\int_0^\infty\mu_t^\eps(s)\l\psi(s),\ptt p\r \d s.
\end{align}

\subsection*{IV. The integral inequality}
Plugging \eqref{est1-Psi} and \eqref{est2-Psi}
into \eqref{PRIMO}, and integrating on $[a,b]$, we are led to
\begin{align*}
&\Psi_\eps(b)+\int_a^b \big[2-R_\eps(t)\big]\|\pt p(t)\|^2\d t\\
&\quad\leq  \Psi_\eps(a)-M\int_a^b
\int_0^\infty[\dot\mu_t(s)+\mu_t'(s)]
\|\psi^t(s)\|^2_1\d s\,\d t
+C\int_a^b \kappa(t)\|\psi^t \|^2_{\M_t}\d t\\
&\quad\quad
+ D\int_a^b \int_0^\infty r_\eps(t,s)\d s\,\d t
-\int_a^b \frac{2\sqrt{R_\eps(t)}\,}{\kappa(t)}
\int_0^\infty\mu_t^\eps(s)\l\psi^t(s),\ptt p(t)\r\d s\,\d t.
\end{align*}
The next step is to let $\eps\to 0$, so to obtain
\begin{align}
\label{est3-Psi}
&\Psi(b)+\int_a^b\|\pt p(t)\|^2\d t\\
\notag
&\quad\leq \Psi(a)-M\int_a^b
\int_0^\infty[\dot\mu_t(s)+\mu_t'(s)]
\|\psi^t(s)\|^2_1\d s\,\d t+C\int_a^b \kappa(t)\|\psi^t \|^2_{\M_t}\d t\\
\notag
&\quad\quad
-\int_a^b \frac{2}{\kappa(t)}
\int_0^\infty\mu_t(s)\l\psi^t(s),\ptt p(t)\r\d s\,\d t.
\end{align}
This will follow from a repeated use of the
Dominated Convergence Theorem (DCT).

\smallskip
\noindent
{\bf -} We first show that $\Psi_\eps(t)\to\Psi(t)$ as $\eps\to 0$.
Indeed, for any fixed $t\in [a,b]$,
it is apparent that $\kappa_\eps(t)\to\kappa(t)$, and
$$\mu_t^\eps(s)\l\psi^t(s),\pt p(t)\r\to \mu_t(s)\l\psi^t(s),\pt p(t)\r,\quad\forall s>0.$$
Since
$$\mu_t^\eps(s)|\l\psi^t(s),\pt p(t)\r|\leq
\mu_t(s)\|\psi^t(s)\|\|\pt p(t)\|\in L^1(\R^+),
$$
the claim follows from the DCT.

\smallskip
\noindent
{\bf -} For every fixed $t\in [a,b]$, we have the convergence $R_\eps(t)\to 1$.
Keeping in mind \eqref{PRIMOR}, from the DCT we easily infer that
$$\int_a^b \big[2-R_\eps(t)\big]\|\pt p(t)\|^2\d t\to \int_a^b \|\pt p(t)\|^2\d t.$$

\smallskip
\noindent
{\bf -}
Similarly, by applying the DCT on $[a,b]\times\R^+$, we prove that
$$\int_a^b \frac{2\sqrt{R_\eps(t)}\,}{\kappa(t)}
\int_0^\infty\mu_t^\eps(s)\l\psi^t(s),\ptt p(t)\r\d s\,\d t
\to \int_a^b \frac{2}{\kappa(t)}
\int_0^\infty\mu_t(s)\l\psi^t(s),\ptt p(t)\r\d s\,\d t.$$
Indeed, on account of {\bf (M5)}, \eqref{PRIMOR} together with
the control $\ptt p\in L^\infty(a,b;\h)$,
$$\bigg| \frac{2\sqrt{R_\eps(t)}\,}{\kappa(t)}
\mu_t^\eps(s)\l\psi^t(s),\ptt p(t)\r\bigg|\leq D\mu_t(s)\|\psi^t(s)\|_1,$$
and
$$\int_0^\infty\mu_t(s)\|\psi^t(s)\|_1\d s\leq \sqrt{\kappa(t)}\,\|\psi^t\|_{\M_t}\in L^1(a,b),$$
due to Remark \ref{int-eta}.

\smallskip
\noindent
{\bf -} We are left to show that
$$\int_a^b \int_0^\infty r_\eps(t,s)\d s\,\d t\to 0.$$
Recalling \eqref{ERREEPS},
we immediately draw the pointwise convergence
$$r_\eps(t,s)\to 0,\quad\forall (t,s)\in [a,b]\times\R^+.$$
By virtue of \eqref{DERPSI} and the hypotheses on $p$ and $\psi_\tau$,
we learn that
$$\|\ps\psi^t\|^2_{L^\infty(0,2;\h)}\leq D,\quad\forall t\in[a,b].$$
Therefore, as $\psi^t\in\D(\T_\tau)$,
$$
\|\psi^t(s)\|\leq \int_0^s\|\partial_s \psi^t(y)\|\d y
\leq Ds,\quad\forall s\in(0,2].
$$
This allows us to obtain the estimate in $[a,b]\times\R^+$
$$
r_\eps(t,s)\leq
\frac{Ds}{\eps}\chi_{[\eps,2\eps]}(s)\mu_t(s)+D\mu_t(s)\|\psi^t(s)\|
\leq D\mu_t(s)+D\mu_t(s)\|\psi^t(s)\|_1.
$$
We already saw that
$$(t,s)\mapsto \mu_t(s)\|\psi^t(s)\|_1\in L^1([a,b]\times\R^+).$$
In order to conclude,
we exploit {\bf (M2)} to get
$$\int_a^b \int_0^\infty \mu_t(s)\d s\,\d t
\leq \kappa(\tau)\int_a^b K_\tau(t)\d t<\infty.
$$
Hence, the DCT applies.

\subsection*{V. Conclusion of the proof}
Using equation \eqref{sis-p-psi},
for every $t\in [a,b]$ we have
\begin{align*}
&-\frac{2}{\kappa(t)}
\int_0^\infty\mu_t(s)\l\psi(s),\ptt p\r \d s\\
&\quad=\frac{2}{\kappa(t)}
\int_0^\infty\mu_t(s)\l\psi(s),p\r_1\d s+ \frac{2}{\kappa(t)}
\bigg\|\int_0^\infty\mu_t(s)A^{1/2}\psi(s)\d s\bigg\|^2\\
&\quad\quad+\frac{2}{\kappa(t)}
\int_0^\infty\mu_t(s)\l\psi(s),\gamma\r\d s.
\end{align*}
Then, for any $\varpi\in (0,1]$, we readily deduce from
{\bf (M5)} the estimate
$$-\frac{2}{\kappa(t)}
\int_0^\infty\mu_t(s)\l\psi(s),\ptt p\r \d s
\leq\varpi\|p\|_1^2
+\frac{C}{\varpi}\kappa(t)\|\psi\|^2_{\M_t}
+\frac{2}{\kappa(t)}
\int_0^\infty\mu_t(s)\l\psi(s),\gamma\r\d s.
$$
Integrating the inequality on $[a,b]$, and
substituting the result into~\eqref{est3-Psi}, the proof of Lemma
\ref{lemma-Psi} is finished.
\qed

\section*{Appendix B\\ A Gronwall-Type Lemma in Integral Form}

\theoremstyle{plain}
\newtheorem{lemmaAPP}{Lemma}[section]
\renewcommand{\thelemmaAPP}{B.\arabic{lemmaAPP}}
\setcounter{equation}{0}
\setcounter{subsection}{0}
\renewcommand{\theequation}{B.\arabic{equation}}

\begin{lemmaAPP}
\label{lemma-new-gw}
Let $\tau\in\R$ be fixed, and let $\Lambda:[\tau,\infty)\to\R$ be a continuous
function. Assume that for some $\eps>0$ and every
$b>a\geq\tau$ the integral inequality
\begin{equation}
\label{phi1}
\Lambda(b) + 2\eps\int_a^b\Lambda(y)\d y
\leq \Lambda(a)+\int_a^b q_1(y)\Lambda(y)\d y + \int_a^b q_2(y)\d y
\end{equation}
holds, where $q_1,q_2$
are locally summable nonnegative functions on $[\tau,\infty)$ satisfying
$$
\int_a^b q_1(y)\d y \leq \eps (b-a)+c_1
\and\sup_{t\geq\tau}\int_t^{t+1}q_2(y)\d y \leq c_2,
$$
for some $c_1,c_2\geq 0$.
Then
$$\Lambda(t)\leq \e^{c_1}|\Lambda(\tau)|\e^{-\eps(t-\tau)}+\frac{c_2\e^{c_1}\e^{\eps}}{1-\e^{-\eps}}$$
for every $t\geq\tau$.
\end{lemmaAPP}

\begin{proof}
For $\nu\in (0,\eps)$ arbitrarily fixed,
we consider for $t\geq\tau$ the continuous positive function
$$
\Pi(t)=\big(|\Lambda(\tau)|+\nu\big)E(t,\tau)+\int_\tau^t q_2(y)E(t,y)\d y,
$$
having defined
$$E(t,y)=\exp{\bigg[-(2\eps-\nu)(t-y)+\int_y^t q_1(z)\d z\bigg]}.$$
Note that
\begin{equation}
\label{monyooo}
E(t,y)\leq\e^{c_1}\e^{(\nu-\eps)(t-y)}, \quad\forall t> y\geq\tau.
\end{equation}
First, we prove the inequality
\begin{equation}
\label{ineq-gw}
\Lambda(t)< \Pi(t),\quad \forall t\geq \tau.
\end{equation}
To this aim, observe that $\Lambda(\tau)<\Pi(\tau)$. Moreover,
it is easily seen that for every $b> a\geq\tau$
\begin{equation}
\label{h1}
\Pi(b)+(2\eps-\nu)\int_a^b\Pi(y)\d y=\Pi(a)+\int_a^b q_1(y)\Pi(y)\d y+\int_a^b q_2(y)\d y.
\end{equation}
We introduce the set
$$\mathcal{O}=\big\{t>\tau: \Lambda(t)\geq \Pi(t)\big\}.$$
By contradiction, let $\mathcal{O}$ be nonempty. Then, defining
$$T=\inf \mathcal{O}<\infty,$$
due to the continuity of $\Lambda$ and $\Pi$ the following hold:
\begin{itemize}
\item[{\rm (i)}] $T>\tau$.
\item[{\rm (ii)}] $\Lambda(T)=\Pi(T)$.
\item[{\rm (iii)}] $\Lambda(t)<\Pi(t)$ for every $t\in[\tau,T)$.
\end{itemize}
In particular, for any $t\in[\tau,T)$,
\begin{equation}
\label{in1}
\Lambda(T)-\Lambda(t)>\Pi(T)-\Pi(t).
\end{equation}
On the other hand, appealing to \eqref{phi1} and \eqref{h1} with $a=t$ and $b=T$, we deduce
\begin{align*}
\Lambda(T)-\Lambda(t)&\leq -2\eps\int_t^T \Lambda(y)\d y+\int_t^T q_1(y)\Lambda(y)\d y
+\int_t^T q_2(y)\d y,\\
\Pi(T)-\Pi(t)&=-(2\eps-\nu) \int_t^T \Pi(y)\d y+\int_t^T q_1(y)\Pi(y)\d y+\int_t^T q_2(y)\d y.
\end{align*}
Therefore, plugging these relationships into \eqref{in1} we obtain
$$-2\eps\int_t^T \Lambda(y)\d y+\int_t^T q_1(y)\Lambda(y)\d y
>-(2\eps-\nu) \int_t^T \Pi(y)\d y+\int_t^T q_1(y)\Pi(y)\d y,$$
and owing to {\rm (iii)} we end up with
$$-2\eps\frac{1}{T-t}\int_t^T \Lambda(y)\d y>-(2\eps-\nu)\frac{1}{T-t}\int_t^T \Pi(y)\d y.$$
At this point, since $\Lambda$ and $\Pi$ are continuous (and equal at $T$),
we can pass to the limit as $t\to T$, so obtaining
$$(2\eps-\nu) \Pi(T)\geq 2\eps\Lambda(T)=2\eps\Pi(T).$$
Since $\Pi(T)>0$, we reach the contradiction
$$2\eps\leq2\eps-\nu.$$
At this point, writing \eqref{ineq-gw} explicitly, for every $t\geq\tau$ we draw
$$\Lambda(t)\leq
\big(|\Lambda(\tau)|+\nu\big)E(t,\tau)+\int_\tau^t q_2(y)E(t,y)\d y.$$
Making use of \eqref{monyooo}, we get
$$\Lambda(t)\leq
\e^{c_1}\big(|\Lambda(\tau)|+\nu\big)\e^{(\nu-\eps)(t-\tau)}
+\e^{c_1}\e^{(\nu-\eps)t}\int_\tau^t q_2(y)\e^{(\eps-\nu)y}\d y.$$
Arguing analogously as in \cite[Theorem 4.1]{PPV}, we estimate
$$\e^{(\nu-\eps)t}\int_\tau^t q_2(y)\e^{(\eps-\nu)y}\d y
\leq \frac{c_2\e^{\eps-\nu}}{1-\e^{\nu-\eps}},$$
providing
$$\Lambda(t)\leq
\e^{c_1}\big(|\Lambda(\tau)|+\nu\big)\e^{(\nu-\eps)(t-\tau)}
+\frac{c_2\e^{c_1}\e^{\eps-\nu}}{1-\e^{\nu-\eps}}.$$
As $\nu\in (0,\eps)$ is arbitrary, a final limit
$\nu\to 0$ gives
$$\Lambda(t)\leq\e^{c_1}|\Lambda(\tau)|\e^{-\eps(t-\tau)}
+\frac{c_2\e^{c_1}\e^{\eps}}{1-\e^{-\eps}},$$
as claimed.
\end{proof}



\end{document}